\documentclass[leqno]{siamltex704}
\usepackage{hyperref}
\usepackage{amsmath}
\usepackage{amsfonts,amssymb}
\usepackage{dsfont}
\usepackage{pifont}
\renewcommand{\ldots}{\dotsc}
\newtheorem{algorithm}{Algorithm}
\newtheorem{model-problem}{Problem}

\newcommand{\bu}{\textbf{u}}
\newcommand{\bw}{\textbf{w}}

\newcommand{\bx}{\textbf{x}}

\newcommand{\bz}{\textbf{z}}

\newcommand{\be}{\textbf{e}}
\newcommand{\bv}{\textbf{v}}

\newcommand{\curl}{{\nabla\times}}

\newcommand{\bpsi}{{\bf \psi}}
\newcommand{\diameter}{\operatorname{diam}}

\def\T{{\mathcal T}}
\def\E{{\mathcal E}}

\def\pT{{\partial T}}

\def\jump#1{{[\![#1[\!]}}
\def\bbf{{\bf f}}
\def\bg{{\bf g}}
\def\bn{{\bf n}}

\def\bU{{\mathbf U}}
\def\bQ{{\mathbf Q}}
\def\bV{{\mathbf V}}

\def\bbR{{\mathbb{R}}}

\def\3bar{{|\hspace{-.02in}|\hspace{-.02in}|}}

\def\leqC{\lesssim}
\def\Hmudiv#1{{H({\rm div}_\mu; #1)}}
\def\Hzmudiv#1{{H_0({\rm div}_\mu; #1)}}

\def\bpsi{{\boldsymbol\psi}}
\def\bvarphi{{\boldsymbol\varphi}}
\def\boldeta{{\boldsymbol\eta}}
\def\bomega{{\boldsymbol\omega}}

\def\bzeta{{\boldsymbol\zeta}}
\def\bchi{{\boldsymbol\chi}}
\def\bphi{{\boldsymbol\phi}}
\def\FmuOmega{{\mathds{F}_\mu(\Omega)}}

\begin{document}

\setlength{\parindent}{0.25in} \setlength{\parskip}{0.08in}

\title{Discretization of div-curl Systems by Weak Galerkin Finite Element Methods
on Polyhedral Partitions}

\author{
Chunmei Wang\thanks{School of Mathematics, Georgia Institute of
Technology, Atlanta, GA 30332 (cwang462@math.gatech.edu). Nanjing
Normal University Taizhou College, Taizhou 225300, China.} \and
Junping Wang\thanks{Division of Mathematical Sciences, National
Science Foundation, Arlington, VA 22230 (jwang@nsf.gov). The
research of Wang was supported by the NSF IR/D program, while
working at National Science Foundation. However, any opinion,
finding, and conclusions or recommendations expressed in this
material are those of the author and do not necessarily reflect the
views of the National Science Foundation.}}

\maketitle

\begin{abstract}
In this paper, the authors devise a new discretization scheme for
div-curl systems defined in connected domains with heterogeneous
media by using the weak Galerkin finite element method. Two types of
boundary value problems are considered in the algorithm development:
(1) normal boundary condition, and (2) tangential boundary
condition. A new variational formulation is developed for the normal
boundary value problem by using the Helmholtz decomposition which
avoids the computation of functions in the harmonic fields. Both
boundary value problems are reduced to a general saddle-point
problem involving the curl and divergence operators, for which the
weak Galerkin finite element method is devised and analyzed. The
novelty of the technique lies in the discretization of the
divergence operator applied to vector fields with heterogeneous
media. Error estimates of optimal order are established for the
corresponding finite element approximations in various discrete
Sobolev norms.

\end{abstract}



\section{Introduction}

\smallskip

This paper is concerned with new developments of numerical methods
for div-curl systems with two types of boundary conditions. The
model problem seeks an unknown function  $\bu=\bu(\bx)$ satisfying
\begin{eqnarray}
\nabla\cdot(\mu\bu)&=&f, \quad\text{in}\ \Omega,\label{div-eqn}\\
\nabla\times\bu &=&\bg,\quad\text{in}\ \Omega,\label{curl-eqn}
\end{eqnarray}
where $\Omega$ is an open bounded and connected domain in
$\mathbb{R}^3$ with a Lipschitz continuous boundary
$\Gamma=\partial\Omega$. The Lebesgue-integrable real-valued
function $f=f(\bx)$ and a vector field $\bg=\bg(\bx)$ are given in
the domain $\Omega$. Here $\mu=\{\mu_{ij}(\bx)\}_{3\times 3}$ is a
symmetric matrix, uniformly positive definite in $\Omega$ and with
entries in $L^\infty(\Omega)$. Assume that the boundary $\Gamma$ has
$m+1$ connected components $\Gamma_i$ such that
$$
\Gamma
=\bigcup_{i=0}^m\Gamma_i,
$$
where $\Gamma_0$ represents the exterior boundary of $\Omega$, and
$\Gamma_i, 1\le i \le m$, the other connected components of
$\Gamma$.

The system (\ref{div-eqn})-(\ref{curl-eqn}) arises in fluid
mechanics and electromagnetic field theories. In  the  fluid
mechanics field theory, the coefficient matrix $\mu(\bx)$ is
diagonal where diagonal entries are the local mass density. In
electrostatics field theory, $\mu(\bx)$ is the permittivity matrix.
In linear magnetic field theory, the function $f(\bx)$ is zero,
$\bu$ represents the magnetic field intensity and $\mu(\bx)$ is the
inverse of the magnetic permeability tensor.

We consider two types of boundary conditions for the div-curl system
(\ref{div-eqn})-(\ref{curl-eqn}): the normal boundary condition, and
the tangential boundary condition.

The normal boundary condition is concerned with a given flux value
for the vector field $\mu\bu$ on $\Gamma$; i.e.,
\begin{equation}\label{flux-bc}
\left(\mu\bu\right)\cdot\bn = \xi,\qquad\mbox{on\ $\Gamma$},
\end{equation}
where  $\textbf{n}$  is the unit outward normal direction on
 $\Gamma$.

The tangential boundary condition corresponds to a given value for
the tangential component of the vector field $\bu$; i.e.,
\begin{equation}\label{tangential-bc}
\begin{split}
\bu\times\bn &= \bchi,\qquad\mbox{ on\ $\Gamma$},\\
\langle \mu\bu\cdot\bn_i, 1\rangle_{\Gamma_i} &=\beta_i,\qquad
i=1,\ldots, m,
\end{split}
\end{equation}
where  $\textbf{n}_i$  is the unit outward normal direction on the
connected component $\Gamma_i$.

The space of harmonic fields is given by
$$
\mathds{H}_{\mu}(\Omega)=\{\bv\in [L^2(\Omega)]^3: \ \curl\bv=0,\
\nabla\cdot(\mu \bv)=0,\ \mu\bv\cdot \bn = 0 \mbox{ on } \Gamma\}.
$$
The space of harmonic fields is non-trivial when the domain is not
simply connected. It is readily seen that the div-curl system
(\ref{div-eqn})-(\ref{curl-eqn}) with the normal boundary condition
(\ref{flux-bc}) is generally not well-possed, as uniqueness fails
(adding a harmonic field to a solution $\bu$ still makes a
solution). For this reason, we shall look for a solution $\bu$ that
is orthogonal to the space of harmonic fields in the $\mu$-weighted
$L^2$ norm. Specifically, the complete problem reads: given $\bg\in
[L^2(\Omega)]^3$ with $\nabla\cdot\bg=0$ in $\Omega$, find $\bu$
such that
\begin{equation}\label{EQ:div-curl:normalBC}
\left\{
\begin{split}
\nabla\cdot(\mu\bu) &= f, \qquad\text{in}\ \Omega,\\
\nabla\times\bu &= \bg,\qquad\text{in}\ \Omega,\\
\mu\bu\cdot\bn & = \xi,\qquad \text{on}\ \Gamma,\\
\int_\Omega \mu\bu\cdot \boldeta & = 0,\qquad \forall\
\boldeta\in\mathds{H}_{\mu}(\Omega).
\end{split}
\right.
\end{equation}
The problem (\ref{EQ:div-curl:normalBC}) has a solution
\cite{saranen}, and uniqueness is straightforward. It is also
well-known that the div-curl system (\ref{div-eqn})-(\ref{curl-eqn})
with the tangential boundary condition (\ref{tangential-bc}) has one
and only one solution.

There have been many numerical methods for approximating div-curl
systems. In \cite{n1992}, Nicolaides proposed and analyzed a control
volume method directly for planar div-curl problems. In
\cite{nw1997}, Nicolaides and Wu presented a special co-finite
volume method for div-curl problems in three dimension, which was
based on a system of two orthogonal grids like the classical
Voronoi-Delaunay mesh pair. Bossavit \cite{b1998} proposed a
classical numerical method for solving the magnetostatic problem by
introducing a scalar or vector potential. In \cite{omnes}, a
discrete duality finite volume method was presented for div-curl
problems on almost arbitrary polygonal meshes. In \cite{bramble},
Bramble and Pasciak proposed and analyzed a direct numerical scheme
under a very weak formulation where the solution space was
$[L^2(\Omega)]^3$. In \cite{copeland}, a mixed finite element method
was analyzed for div-curl systems in simply connected axisymmetric
domains. In \cite{valli2}, a numerical algorithm was designed for
constructing a finite element basis for the first de Rham cohomology
group of the computational domain, which was further used for a
numerical approximation of the magnetostatic problem. In
\cite{brezzi2} and \cite{lipnikov}, the mimetic finite difference
method was applied to the 3D magnetostatic problems on general
polyhedral meshes.

Recently, weak Galerkin (WG) finite element methods have emerged as
a new numerical technique for approximating the solutions of partial
differential equations. The WG method was first introduced in
\cite{wy1202, wy2013} for the second order elliptic problem and was
further developed in \cite{wy1202, mwy1204, wy1302, cwang} with
other applications. Two basic principles for the WG finite element
method are: (1) the differential operators (e.g., gradient,
Laplacian, Hessian, curl, divergence etc.) are interpreted and
approximated as distributions over a set of generalized functions,
and (2) proper stabilizations are employed to enforce necessary weak
continuities for approximating functions in the correct topology. It
has been demonstrated that the WG method is highly flexible and
robust as a numerical tool that makes use of discontinuous piecewise
polynomials on polygonal or polyhedral finite element partitions.

The goal of this paper is to present a new discretization scheme for
the div-curl system (\ref{div-eqn})-(\ref{curl-eqn}) in any
connected domain with heterogeneous media by using the weak Galerkin
finite element approach. In particular, for the normal boundary
value problem (\ref{EQ:div-curl:normalBC}), a new variational
formulation is developed by using a Helmholtz decomposition which
avoids the computation of any harmonic fields $\bv\in
\mathds{H}_\mu(\Omega)$. The resulting formulation is a special case
of the model problem (\ref{EQ:model-problem-I}) detailed in Section
5. The div-curl system with the tangential boundary condition
(\ref{tangential-bc}) is also formulated as a special case of the
model problem (\ref{EQ:model-problem-I}). Therefore, our attention
is focused on the development of weak Galerkin finite element
methods for (\ref{EQ:model-problem-I}). It is readily seen that the
main difficulty in numerical methods for (\ref{EQ:model-problem-I})
lies in the term $\nabla\cdot(\mu \bu)$ which requires the
continuity of $\mu\bu$ in the normal direction of any interface,
particularly the interface of any two polyhedral elements. The weak
Galerkin finite element method offers an ideal solution, as the
continuity can be relaxed by a weak continuity implemented through a
carefully chosen stabilizer.

The paper is organized as follows. In Section 2, we introduce some
commonly used notations. In Section 3, we derive a formulation for
the div-curl problem (\ref{div-eqn})-(\ref{curl-eqn}) with normal
boundary condition (\ref{flux-bc}) by using Helmholtz decomposition.
Section 4 is devoted to a discussion of the div-curl system with
tangential boundary condition. In Section 5, we discuss a model
problem that is central to the solution of the div-curl system with
both the normal and tangential boundary conditions. In Section 6, we
introduce some discrete weak differential operators which are
necessary for the development of weak Galerkin finite element
methods in Section 7. Section 8 is devoted to a discussion of
existence and uniqueness for the solution of the weak Galerkin
discretizations. In Section 9, we derive some error equations. An
{\em inf-sup} condition is established in Section 10. Finally in
Section 11, we derive some optimal order error estimates for the WG
finite element approximations.

\section{Notations and Preliminaries}

Throughout the paper, we will follow the usual notation for Sobolev
spaces and norms \cite{ciarlet}. For any open bounded domain
$D\subset \mathbb{R}^3$ with Lipschitz continuous boundary, we use
$\|\cdot\|_{s,D}$ and $|\cdot|_{s,D}$ to denote the norm and
seminorm in the Sobolev space $H^s(D)$ for any $s\ge 0$,
respectively. The inner product in $H^s(D)$ is denoted by
$(\cdot,\cdot)_{s,D}$. The space $H^0(D)$ coincides with $L^2(D)$,
for which the norm and the inner product are denoted by $\|\cdot
\|_{D}$ and $(\cdot,\cdot)_{D}$, respectively.

Let $\mu=\{\mu_{ij}\}_{3\times 3}$ be a symmetric matrix, uniformly
positive definite in $D$ and with entries in $L^\infty(D)$.
Introduce the following Sobolev space
\begin{equation*}
\Hmudiv{D}=\{\bv\in [L^2(D)]^3: \ \nabla\cdot(\mu\bv)\in L^2(D)\},
\end{equation*}
with norm given by
$$
\|\bv\|_{\Hmudiv{D}} = (\|\bv\|^2_D+\|\nabla\cdot (\mu
\bv)\|^2_D)^{\frac12},
$$
where $\nabla\cdot(\mu\bv)$ is the divergence of $\mu\bv$. Any
$\bv\in \Hmudiv{D}$ can be assigned a trace for the normal component
of $\mu\bv$ on the boundary. The subspace with vanishing trace in
the normal component is denoted by
$$
\Hzmudiv{D}=\{\bv\in \Hmudiv{D}: \ (\mu\bv)\cdot\bn|_{\partial D} =
0\}.
$$
Denote the subspace of $\Hzmudiv{D}$ with divergence-free vectors by
$$
\mathds{F}_\mu(D)=\{\bv\in H({\rm div}_\mu; D): \
\nabla\cdot(\mu\bv) = 0\}.
$$
When $\mu=I$ is the identity matrix, the spaces $\Hmudiv{D}$,
$\Hzmudiv{D}$, and $\mathds{F}_\mu(D)$ shall be denoted as $H({\rm
div}; D)$, $H_0({\rm div}; D)$, and $\mathds{F}(D)$, respectively.

Denote by $H({\rm curl}; D)$ the following Sobolev space
$$
H({\rm curl}; D)=\{\bv: \bv\in [L^2(D)]^3, \nabla \times \bv\in
[L^2(D)]^3\}
$$
with norm defined by
$$
\|\bv\|_{H({\rm curl}; D)}=(\|\bv\|^2_D+\|\nabla\times
\bv\|^2_D)^{\frac12},
$$
where $\nabla\times\bv$ is the curl of $\bv$. Any $\bv\in H({\rm
curl}; D)$ can be assigned a trace for its tangential component on
the boundary. The subspace of $H({\rm curl}; D)$ with vanishing
trace in the tangential component is denoted by
$$
H_0({\rm curl}; D)=\{\bv\in H({\rm curl}; D): \
\bv\times\bn|_{\partial D} = 0 \}.
$$

When $D=\Omega$, we shall drop the subscript $D$ in the norm and
inner product notation.

For simplicity of notation, throughout the paper, we use
``$\lesssim$ '' to denote ``less than or equal to up to a general
constant independent of the mesh size or functions appearing in the
inequality".

\section{The div-curl System with Normal Boundary Condition}
The goal of this section is to derive a suitable variational
formulation for the problem (\ref{EQ:div-curl:normalBC}). Denote by
$H^0({\rm curl};\Omega)=\{\bv\in [L^2(\Omega)]^3: \ \curl\bv=0\}$
the space of curl-free fields. It is well-known that any vector
field $\bv\in H^0({\rm curl};\Omega)$ can be written as (see, e.g.,
\cite{valli3})
\begin{equation}\label{EQ:curl-free-by-gradient+}
\bv = \nabla \phi + \boldeta,
\end{equation}
where $\phi\in H^1(\Omega)$ and $\boldeta\in
\mathds{H}_{\mu}(\Omega)$. It follows that
$$
(\nabla\phi, \mu\boldeta) = -(\phi,
\nabla\cdot(\mu\boldeta))+\langle \phi,
\mu\boldeta\cdot\bn\rangle_\Gamma = 0.
$$
The decomposition (\ref{EQ:curl-free-by-gradient+}) is thus
orthogonal in the $\mu$-weighted $L^2$ norm in $H^0({\rm
curl};\Omega)$.

\subsection{Helmholtz decomposition}
Denote by $H_{0c}^1(\Omega)$ the set of functions in $H^1(\Omega)$
with vanishing value on $\Gamma_0$ and constant values on other
connected components of the boundary; i.e.,
\begin{equation*}\label{EQ:Nov-11-2014:H0c}
H_{0c}^1(\Omega)=\{\phi\in H^1(\Omega):\ \phi|_{\Gamma_0}=0, \
\phi|_{\Gamma_i}=c, \ i=1,\ldots, m\}.
\end{equation*}
Next, introduce a Sobolev space
\begin{equation*}
\begin{array}{c}
\mathds{Y}_\mu(\Omega) = \{\bv\in H_0({\rm curl};\Omega): \ \langle
\mu\bv\cdot\bn_i, 1\rangle_{\Gamma_i} = 0, \ i=1,\ldots, m\}.
\end{array}
\end{equation*}
The following Helmholtz decomposition holds the key to the
derivation of a suitable variational form for the problem
(\ref{EQ:div-curl:normalBC}).

\begin{theorem}\label{THM:helmholtz-2}
For any vector-valued function $\bu\in [L^2(\Omega)]^3$, there
exists a unique $\bpsi\in \mathds{Y}_\mu(\Omega)\cap \FmuOmega,\
\phi\in H^1(\Omega)/\mathbb{R}$, and $\boldeta\in
\mathds{H}_\mu(\Omega)$ such that
\begin{equation}\label{EQ:helmholtz-2}
\bu = \mu^{-1}\nabla\times\bpsi + \nabla\phi + {\boldsymbol\eta}.
\end{equation}
Moreover, the following estimate holds true
\begin{equation}\label{EQ:helmholtz-288}
\|\bpsi\|_{H({\rm curl}; \Omega)} + \|\nabla\phi\|_0 \leqC
(\kappa\bu,\bu)^{\frac12}.
\end{equation}
\end{theorem}

\begin{proof}
Consider the problem of seeking $\bpsi\in \mathds{Y}_\mu(\Omega)\cap
\FmuOmega$ such that
\begin{equation}\label{EQ:May20-100}
(\mu^{-1}\nabla\times\bpsi, \nabla\times\bvarphi)=(\bu,
\nabla\times\bvarphi),\qquad \forall\ \bvarphi\in
\mathds{Y}_\mu(\Omega)\cap \FmuOmega.
\end{equation}
Denote by
$$
a(\bpsi,\bvarphi):=(\mu^{-1}\nabla\times\bpsi, \nabla\times\bvarphi)
$$
the bilinear form defined on $\mathds{Y}_\mu(\Omega)\cap \FmuOmega$.
We claim that $a(\cdot,\cdot)$ is coercive with respect to the
$H({\rm curl}; \Omega)$-norm. To this end, it suffices to derive the
following estimate
\begin{equation}\label{EQ:May20-101}
\|\bv\|_0 \lesssim \|\nabla\times\bv\|_0,\qquad\forall \bv\in
\mathds{Y}_\mu(\Omega)\cap \FmuOmega.
\end{equation}
In fact, for any $\bv\in \mathds{Y}_\mu(\Omega)\cap \FmuOmega$, from
Theorem 3.4 of \cite{girault-raviart}, there exists a vector
potential function $\bomega\in [H^1(\Omega)]^3$ such that
\begin{equation}\label{ForCM}
\mu\bv = \nabla\times\bomega,\ \nabla\cdot\bomega=0,\
\|\bomega\|_1\leqC (\mu\bv,\bv)^{\frac12}.
\end{equation}
Using the integration by parts and the condition $\bv\times\bn=0$ on
$\Gamma$, we have
$$
(\mu\bv,\bv) = (\nabla\times\bomega, \bv)=(\bomega,\nabla\times\bv).
$$
It follows from the Cauchy-Schwarz inequality and (\ref{ForCM}) that
$$
(\mu\bv,\bv) \leq \|\bomega\|_0\ \|\nabla\times\bv\|_0 \lesssim
(\mu\bv,\bv)^{\frac12}\ \|\nabla\times\bv\|_0,
$$
which implies (\ref{EQ:May20-101}).

Now from the Lax-Milgram Theorem, there exists a unique $\bpsi\in
\mathds{Y}_\mu(\Omega)\cap \FmuOmega$ satisfying the equation
(\ref{EQ:May20-100}) such that
$$
\|\bpsi\|_{H({\rm curl}; \Omega)}\lesssim \|\bu\|_0\lesssim (\mu
\bu,\bu).
$$
It is easy to see that $\mathds{Y}_\mu(\Omega)\cap \FmuOmega$ is
equivalent to the following quotient space:
$$
H_0({\rm curl};\Omega)/(\nabla H_{0c}^1(\Omega)) = \{\bv\in H_0({\rm
curl};\Omega):\ (\mu\bv, \nabla\phi)=0,\ \forall \phi\in
H_{0c}^1(\Omega) \}.
$$
Thus, by using a Lagrange multiplier $p\in H_{0c}^1(\Omega)$, the
problem (\ref{EQ:May20-100}) can be re-formulated as follows: Find
$\bpsi\in H_0({\rm curl};\Omega)$ and $p\in H_{0c}^1(\Omega)$ such
that
\begin{equation}\label{EQ:May20-102}
\begin{split}
(\mu^{-1}\nabla\times\bpsi, \nabla\times\bvarphi)+(\mu\nabla p,
\bvarphi)&=(\bu,
\nabla\times\bvarphi),\quad \forall\ \bvarphi\in H_0({\rm curl};\Omega),\\
(\bpsi, \mu\nabla s)&=0,\quad\qquad\qquad\forall\ s\in
H_{0c}^1(\Omega).
\end{split}
\end{equation}
It follows from the first equation of (\ref{EQ:May20-102}) that
$$
\nabla\times(\bu - \mu^{-1}\nabla\times \bpsi) - \mu\nabla p = 0.
$$
Since $p\in H_{0c}^1(\Omega)$, then the two terms on the left-hand
side of the above equation are orthogonal in the $\mu^{-1}$-weighted
$L^2(\Omega)$ norm. Thus,
$$
\nabla\times(\bu - \mu^{-1}\nabla\times \bpsi) =0 \Longrightarrow \
\bu - \mu^{-1}\nabla\times \bpsi \in H^0({\rm curl};\Omega).
$$
Thus, there exist unique $\phi\in H^1(\Omega)/\mathbb{R}$ and
$\boldeta\in \mathds{H}_\mu(\Omega)$ such that
$$
\bu - \kappa^{-1}\nabla\times \bpsi =\nabla \phi + \boldeta,
$$
which completes the proof of the theorem.
\end{proof}

\subsection{A variational formulation}\label{Section:div-curl-normalBC}

Assume that the div-curl problem (\ref{div-eqn})-(\ref{curl-eqn})
with boundary condition (\ref{flux-bc}) has a solution. Integrating
(\ref{div-eqn}) over the domain $\Omega$ and from the integration by
parts,  we have
$$
(f, 1) = (\nabla\cdot(\mu \bu), 1) = \langle (\mu\bu) \cdot\bn,
1\rangle_\Gamma.
$$
Using the boundary condition (\ref{flux-bc}), we arrive at the
following compatibility condition
\begin{equation}\label{Type-I-Compatibility-1}
(f, 1) =\langle \xi, 1\rangle_\Gamma.
\end{equation}
In addition,  taking the divergence to the equation
(\ref{curl-eqn}), we obtain a second compatibility condition
\begin{equation}\label{Type-I-Compatibility-2}
\nabla\cdot\bg=0\qquad\mbox{in} \ \Omega.
\end{equation}

\begin{lemma}\label{Lemma:Decomposition-type-I-bc}
Let $\bu$ be a solution of the div-curl system
(\ref{div-eqn})-(\ref{curl-eqn}) with the boundary condition
(\ref{flux-bc}). Then, $\bu$ can be decomposed as follows
\begin{equation}\label{EQ:section2:101}
\bu = \mu^{-1}\nabla\times\bpsi + \nabla\phi + \boldeta,
\end{equation}
where $\bpsi\in \mathds{Y}_\mu(\Omega)\cap \FmuOmega$ is the unique
solution of the following equation
\begin{equation}\label{EQ:decom:BC-I:001}
(\mu^{-1}\curl\bpsi,\curl\bzeta) =(\bg,\bzeta),\qquad\forall\
\bzeta\in \mathds{Y}_\mu(\Omega)\cap \FmuOmega,
\end{equation}
and $\phi\in H^1(\Omega)/\mathbb{R}$ satisfies
\begin{eqnarray}\label{EQ:decom:BC-I:002}
(\mu\nabla\phi,\nabla v) =\langle \xi, v\rangle - (f,
v),\qquad\forall \ v\in H^1(\Omega)/\mathbb{R}.
\end{eqnarray}
\end{lemma}

\begin{proof}
Using the Helmholtz decomposition (\ref{EQ:helmholtz-2}) in Theorem
\ref{THM:helmholtz-2}, there exist unique $\bpsi\in
\mathds{Y}_\mu(\Omega)\cap \FmuOmega$, $\phi\in
H^1(\Omega)/\mathbb{R}$, and $\boldeta\in \mathds{H}_\mu(\Omega)$
such that
\begin{equation}\label{EQ:November-14:001}
\bu = \mu^{-1} \nabla\times \bpsi + \nabla\phi +\boldeta.
\end{equation}
For any $\bzeta\in \mathds{Y}_\mu(\Omega)\cap \FmuOmega$,
$$
(\nabla\phi, \curl\bzeta) = (\curl\nabla\phi,
\bzeta)-\langle\nabla\phi, \bzeta\times\bn\rangle_\Gamma =0.
$$
Thus, by testing both sides of (\ref{EQ:November-14:001}) with
$\curl\bzeta$, we obtain
\begin{equation}\label{EQ:November-14:002}
(\bu-\boldeta,\curl\bzeta)=(\mu^{-1}\curl\bpsi,\curl\bzeta).
\end{equation}
From the equation (\ref{curl-eqn}) and the fact that
$\curl\boldeta=0$ and $\bzeta\times\bn=0$ on $\Gamma$, we obtain
$$
(\bu-\boldeta,\curl\bzeta) = (\curl\bu, \bzeta) =(\bg, \bzeta).
$$
Substituting the above into (\ref{EQ:November-14:002}) yields
\begin{equation}\label{EQ:November-14:003}
(\mu^{-1}\curl\bpsi,\curl\bzeta)=(\bg, \bzeta),\qquad\forall \
\bzeta\in \mathds{Y}_\mu(\Omega)\cap \FmuOmega,
\end{equation}
which verifies (\ref{EQ:decom:BC-I:001}).

Next, we test (\ref{EQ:November-14:001}) against any
$\mu\nabla\varphi$ to obtain
\begin{equation}\label{EQ:November-14:004}
(\bu, \mu\nabla\varphi) = (\mu^{-1}\curl\bpsi,
\mu\nabla\varphi)+(\mu\nabla\phi, \nabla\varphi) + (\boldeta,
\mu\nabla\varphi),\qquad \forall \varphi\in H^1(\Omega).
\end{equation}
Using the integration by parts,
$$
(\bu, \mu\nabla\varphi) = - (\nabla\cdot(\mu\bu), \varphi) + \langle
(\mu\bu)\cdot\bn, \varphi\rangle_\Gamma.
$$
Thus, from the equation (\ref{div-eqn}) and the boundary condition
(\ref{flux-bc}), we have
\begin{equation}\label{EQ:12-21:001}
(\bu, \mu\nabla\varphi) = - (\bbf, \varphi) + \langle \xi,
\varphi\rangle_\Gamma.
\end{equation}
Similarly, from the integration by parts and the fact that
$\boldeta\in \mathds{H}_\mu(\Omega)$,
\begin{equation}\label{EQ:12-21:002}
(\boldeta, \mu\nabla\varphi) = - (\nabla\cdot(\mu\boldeta), \varphi)
+ \langle (\mu\boldeta)\cdot\bn, \varphi\rangle_\Gamma = 0.
\end{equation}
Since $\bpsi\in H_0({\rm curl};\Omega)$, then
\begin{equation}\label{EQ:12-21:003}
(\curl\bpsi, \nabla\varphi) = 0.
\end{equation}
Substituting (\ref{EQ:12-21:001})-(\ref{EQ:12-21:003}) into
(\ref{EQ:November-14:004}) gives rise to
\begin{equation}\label{EQ:November-14:005}
(\mu\nabla\phi, \nabla\varphi) = \langle \xi, \varphi\rangle_\Gamma-
(\bbf, \varphi), \qquad \forall \varphi\in H^1(\Omega).
\end{equation}
The above problem has a unique solution in the quotient space
$H^1(\Omega)/\mathbb{R}$ due to the compatibility condition
(\ref{Type-I-Compatibility-1}). This completes the proof.
\end{proof}

By using a Lagrange multiplier $p$, the problem
(\ref{EQ:decom:BC-I:001}) can be re-formulated as a saddle point
problem that seeks $\bpsi\in \mathds{Y}_\mu(\Omega)\cap H({\rm
div}_\mu;\Omega)$ and $p\in L^2(\Omega)$ satisfying
\begin{equation}\label{eq-100.01}
\begin{split}
(\mu^{-1}\curl\bpsi,\curl\bzeta) + (\nabla\cdot(\mu\bzeta),
p)&=(\bg,\bzeta),\qquad\forall\ \bzeta\in \mathds{Y}_\mu(\Omega)\cap
H({\rm div}_\mu;\Omega),\\
(\nabla\cdot(\mu\bpsi), w)&= 0, \qquad \qquad\forall\ w\in
L^2(\Omega).
\end{split}
\end{equation}

Going back to the well-posed problem (\ref{EQ:div-curl:normalBC}),
it is easily seen that the solution of (\ref{EQ:div-curl:normalBC})
also admits the decomposition (\ref{EQ:section2:101}). Thus, using
$\bpsi\times\bn=0$ and $\mu\boldeta\cdot\bn=0$ on $\Gamma$, we
obtain
\begin{equation*}
\begin{array}{rcll}
0&=&(\mu\bu,\boldeta)\qquad &\mbox{by last condition in (\ref{EQ:div-curl:normalBC})} \\
 &=&(\curl\bpsi, \boldeta)+(\mu\nabla \phi, \boldeta)+ (\mu\boldeta,
 \boldeta)\qquad &\mbox{by the decomposition
 (\ref{EQ:section2:101})}\\
 &=&(\bpsi, \curl\boldeta)-(\phi, \nabla\cdot(\mu\boldeta))+(\mu\boldeta,
 \boldeta)\qquad &\mbox{by integration by parts}\\
 &=&(\mu\boldeta, \boldeta).\qquad &\mbox{by $\boldeta\in \mathds{H}_\mu(\Omega)$}
\end{array}
\end{equation*}
It follows that $\boldeta\equiv 0$. The result can be summarized as
follows.

\begin{theorem} \label{THM:formulation-4-normalBVP}
Let $\bu$ be the solution of the div-curl system
(\ref{EQ:div-curl:normalBC}). Then, $\bu$ can be represented
\begin{equation}\label{EQ:section2:101-new}
\bu = \nabla\times\bpsi + \nabla\phi,
\end{equation}
where $\bpsi\in \mathds{Y}_\mu(\Omega)\cap \FmuOmega$ is the unique
solution of the system of equations (\ref{eq-100.01}), and $\phi\in
H^1(\Omega)$ is determined by the equation
(\ref{EQ:decom:BC-I:002}).
\end{theorem}

Theorem \ref{THM:formulation-4-normalBVP} indicates that a suitable
variational formation for the div-curl system
(\ref{EQ:div-curl:normalBC}) is given by (\ref{eq-100.01}) and
(\ref{EQ:decom:BC-I:002}), which are independently defined and each
has one and only one solution in the corresponding Sobolev space.
The problem (\ref{EQ:decom:BC-I:002}) is a standard second order
elliptic equation for which many existing numerical methods can be
applied. But the problem (\ref{eq-100.01}) requires a study in
numerical methods, which is the central topic of this paper.

\section{The div-curl System with Tangential Boundary Condition}
\label{Section:div-curl-tangentialBC} The goal of this section is to
derive a variational formulation for the div-curl system with the
tangential boundary condition (\ref{tangential-bc}). Recall that the
complete problem reads: given $\bg\in [L^2(\Omega)]^3$ with
$\nabla\cdot\bg=0$ in $\Omega$, find the vector-valued function
$\bu$ such that
\begin{equation}\label{EQ:div-curl:tangentialBC}
\left\{
\begin{split}
\nabla\cdot(\mu\bu) &= f, \qquad\text{in}\ \Omega,\\
\nabla\times\bu &= \bg,\qquad\text{in}\ \Omega,\\
\bu\times\bn & = \bchi,\qquad \text{on}\ \Gamma,\\
\langle \mu\bu\cdot\bn_i, 1\rangle_{\Gamma_i} &=\beta_i,\qquad
i=1,\ldots, m.
\end{split}
\right.
\end{equation}

The problem (\ref{EQ:div-curl:tangentialBC}) can be interpreted as a
constraint minimization problem with an object function given by
$$
J(\bv) = \frac12 \|\curl\bv - \bg\|_0^2,
$$
subject to the following constraints
\begin{equation}\label{EQ:div-curl:constraints}
\left\{
\begin{split}
\nabla\cdot(\mu\bv) &= f, \qquad\text{in}\ \Omega,\\
\bv\times\bn & = \bchi,\qquad \text{on}\ \Gamma,\\
\langle \mu\bv\cdot\bn_i, 1\rangle_{\Gamma_i} &=\beta_i,\qquad
i=1,\ldots, m.
\end{split}
\right.
\end{equation}
By introducing a Lagrange multiplier $p$, the corresponding
variational problem seeks $\bu\in H({\rm curl}; \Omega)\cap H({\rm
div}_\mu;\Omega)$ and $p\in L^2(\Omega)$ with $\bu\times\bn=\bchi$
on $\Gamma$ and $\langle \mu\bu\cdot\bn_i, 1\rangle_{\Gamma_i}
=\beta_i$ for $i=1,\ldots, m$ such that
\begin{equation}\label{eq-800.08}
\begin{split}
(\curl\bu,\curl\bzeta) + (\nabla\cdot(\mu\bzeta),
p)&=(\curl\bg,\bzeta),\qquad\forall\ \bzeta\in
\mathds{Y}_\mu(\Omega)\cap
H({\rm div}_\mu;\Omega),\\
(\nabla\cdot(\mu\bu), w)&= (f,w), \ \quad \qquad\forall\ w\in
L^2(\Omega).
\end{split}
\end{equation}

The problem (\ref{eq-800.08}) is the desired variational form for
the div-curl system with tangential boundary condition
(\ref{tangential-bc}). The structure of this variational problem is
essentially the same as that of (\ref{eq-100.01}) arising from the
normal boundary condition.

\begin{lemma}\label{Lemma:Solution4P1IsUnique}
The solution to the variational problem (\ref{eq-800.08}) is unique
for any connected domain $\Omega$.
\end{lemma}

\begin{proof}
It suffices to show that all the solutions corresponding to the one
with homogeneous data are trivial. To this end, let $(\bu; p)$ be a
solution of (\ref{eq-800.08}) with homogeneous data, then
\begin{eqnarray}
\curl\bu &=& 0, \quad \mbox{in} \
\Omega,\label{EQ:12-14:001}\\
\nabla\cdot(\mu\bu) &=& 0,\quad \mbox{in} \
\Omega,\label{EQ:12-14:002}\\
\langle \mu\bu\cdot\bn_i,1\rangle_{\Gamma_i} &=& 0,\quad i=1,\ldots,
m,\label{EQ:12-14:003} \\
\bu\times\bn &=& 0,\quad \mbox{on}\
\partial\Omega.\label{EQ:12-14:004}
\end{eqnarray}
From (\ref{EQ:12-14:002}) and (\ref{EQ:12-14:003}), the conditions
of Theorem 3.4 in \cite{girault-raviart} are satisfied for the
vector-valued function $\mu \bu$. Thus, there exists a vector
potential $\bphi\in [H^1(\Omega)]^3$ such that $\mu\bu=\curl\bphi$.
It follows that
\begin{equation*}
\begin{split}
(\mu\bu, \bu) & =(\curl\bphi, \bu) \\
& = (\bphi, \curl\bu)+\langle \bn\times\bphi,
\bu\rangle_{\Gamma} \ \ \quad\mbox{by integration by parts} \\
& =\langle \phi, \bu\times\bn\rangle_{\Gamma} \ \quad\qquad\qquad\qquad\mbox{by (\ref{EQ:12-14:001}) and triple product property}\\
& =0,\ \quad\qquad\qquad\qquad\qquad\qquad\mbox{by
(\ref{EQ:12-14:004})}
\end{split}
\end{equation*}
which implies $\bu\equiv 0$. Consequently, one has the following
equation
$$
(\nabla\cdot(\mu\bzeta), p)=0,\qquad \forall \bzeta\in
\mathds{Y}_\mu(\Omega)\cap H({\rm div}_\mu;\Omega),
$$
which leads to $p=0$ by selecting a particular vector field
$\bzeta\in \mathds{Y}_\mu(\Omega)\cap H({\rm div}_\mu;\Omega)$ such
that $\nabla\cdot(\mu\bzeta)=p$. Such a vector field is given by
$\bzeta=\nabla w$, where $w$ is the solution of the following
equation
$$
\nabla\cdot(\mu\nabla w) = p, \quad w|_{\Gamma_0}=0, \quad
w|_{\Gamma_i}=\gamma_i,\ i=1,\ldots, m.
$$
Here, $\{\gamma_i\}$ is a set of real numbers which can be tuned so
that $\langle \mu\nabla w\cdot\bn_i, 1\rangle_{\Gamma_i}=0$.
\end{proof}

\section{A Model Problem}
It is readily seen from Subsections \ref{Section:div-curl-normalBC}
and Section \ref{Section:div-curl-tangentialBC} that the core
problem of study for the div-curl system with either the normal or
the tangential boundary conditions is one structured in the form of
(\ref{eq-100.01}) and (\ref{eq-800.08}). Thus, we shall consider a
general problem in the form of (\ref{eq-100.01}) and
(\ref{eq-800.08}) as follows. Assume the following data are given:
$\bg\in [L^2(\Omega)]^3,\ f\in L^2(\Omega)$, \ $\xi\in
[H^{-\frac12}(\Gamma)]^3$, \ a symmetric and positive definite
matrix $\kappa=\left(\kappa_{ij}(\bx)\right)_{3\times 3}$ in the
domain $\Omega$, and a set of real numbers $\beta_i, i=1,\ldots, m$.

\medskip

\begin{model-problem}
Find $\bu\in H({\rm curl};\Omega) \cap H({\rm div}_\mu;\Omega)$ and
$p\in L^2(\Omega)$ such that
\begin{equation}\label{EQ:model-problem-I}
\left\{
\begin{array}{rll}
(\kappa \curl\bu,\curl\bv) + (\nabla\cdot(\mu\bv), p)= &  (\bg,\bv),
&\quad\forall\ \bv\in
\mathds{Y}_\mu(\Omega)\cap H({\rm div}_\mu;\Omega),\\
(\nabla\cdot(\mu\bu), w) =& (f,w), &\quad \forall\ w\in L^2(\Omega),
\\
\langle \mu\bu\cdot\bn_i, 1\rangle_{\Gamma_i} =&\beta_i, &\quad
i=1,\ldots, m,\\
\bu\times\bn =&\xi, &\quad \mbox{on } \Gamma.
\end{array}
\right.
\end{equation}
Here $\bn$ is the unit outward normal direction to the boundary
$\Gamma$.
\end{model-problem}

The constraint of $\langle \mu\bv\cdot\bn_i, 1\rangle_{\Gamma_i}=0$
for the test space $\mathds{Y}_\mu(\Omega)\cap H({\rm
div}_\mu;\Omega)$ is cumbersome in the design of numerical methods
for the problem (\ref{EQ:model-problem-I}). One possible remedy is
to relax this constraint by using a Lagrange multiplier denoted by
$\lambda=(\lambda_1,\ldots, \lambda_m)$. The corresponding weak
formulation seeks $\bu\in H({\rm curl};\Omega) \cap H({\rm
div}_\mu;\Omega)$, $p\in L^2(\Omega)$, and $\lambda\in \bbR^m$ such
that $\bu\times\bn =\xi$ on $\Gamma$ and
\begin{equation}\label{EQ:model-problem-I-new}
\left\{
\begin{array}{c}\displaystyle
(\kappa \curl\bu,\curl\bv) + (\nabla\cdot(\mu\bv), p) +\sum_{i=1}^m
\langle \mu\bv\cdot\bn_i, \lambda_i\rangle_{\Gamma_i}= (\bg,\bv),\\
\displaystyle(\nabla\cdot(\mu\bu), w) + \sum_{i=1}^m \langle
\mu\bu\cdot\bn_i, s_i\rangle_{\Gamma_i} = (f,w) + \sum_{i=1}^m
\beta_i s_i,
\end{array}
\right.
\end{equation}
for all $\bv\in H_0({\rm curl};\Omega)\cap H({\rm div}_\mu;\Omega)
,\ w\in L^2(\Omega),$ and $s\in \bbR^m$. It is not hard to see that
the problems (\ref{EQ:model-problem-I-new}) and
(\ref{EQ:model-problem-I}) are equivalent to each other.

\medskip
For simplicity of analysis, throughout the paper, we assume that
$\kappa$ and $\mu$ are piecewise constant, symmetric and positive
definite matrices on the domain $\Omega$ with respect to any finite
element partitions to be specified in forthcoming sections.

\section{Weak Differential Operators}

The model problem (\ref{EQ:model-problem-I}) is formulated with two
principle differential operators: divergence and curl. This section
shall introduce the notion of weak divergence operator for
vector-valued functions of the form $\mu\bv$. For completeness, we
also review the definition for the weak curl operator. These weak
differential operators shall be discretized by using polynomials,
which leads to discretizations for the model problem
(\ref{EQ:model-problem-I}).

Let $K\subset\Omega$ be any open bounded domain with boundary
$\partial K$. Denote by $\bn$ the unit outward normal direction on
$\partial K$. Let the space of weak vector-valued functions in $K$
be given by
\begin{eqnarray*}
V(K)=\{\bv=\{\bv_0,\bv_{b}\}: \ \textbf{v}_0\in [L^2(K)]^3,\
\bv_{b}\in [L^2(\partial K)]^3\},
\end{eqnarray*}
where $\textbf{v}_0$ represents the value of $\textbf{v}$ in the
interior of $K$, and $\bv_{b}$ represents certain information of
$\bv$ on the boundary $\partial K$. There are two pieces of
information of $\bv$ on $\partial K$ that are necessary for defining
the variational formulation (\ref{EQ:model-problem-I}): (1) the
tangential component $\bn\times(\bv \times\bn)$, and (2) the normal
component of $\mu\bv$ on $\partial K$. The normal component of the
vector $\mu\bv$ is given by $(\mu\bv\cdot\bn)\bn$. Intuitively
speaking, the vector $\bv_b$ should carry those two pieces of
orthogonal information by summing up $(\mu\bv\cdot\bn)\bn$ and
$\bn\times(\bv\times \bn)$:
$$
\bv_b =(\mu\bv\cdot\bn)\bn +\bn\times(\bv\times \bn).
$$
It is easy to check the following identity:
\begin{equation}\label{interteller}
\bv_b\times\bn = (\bn\times(\bv\times \bn))\times\bn = \bv\times\bn.
\end{equation}

\subsection{Weak divergence}
Following \cite{wy1302}, for any $\textbf{v}\in V(K)$, we define the
weak divergence of $\mu\bv$, denoted by $\nabla_w\cdot
(\mu\textbf{v})$, as a bounded linear functional in the Sobolev
space $H^1(K)$ such that
\begin{equation*}
\langle \nabla_w\cdot
(\mu\textbf{v}),\varphi\rangle_K=-(\mu\textbf{v}_0,\nabla
\varphi)_K+\langle \bv_{b}\cdot\bn,\varphi\rangle_{\partial
K},\qquad \forall \ \varphi\in H^1(K).
\end{equation*}
The discrete weak divergence of $\mu\bv$, denoted by
$\nabla_{w,r,K}\cdot (\mu\textbf{v})$, is defined as the unique
polynomial in $P_r(K)$, satisfying
\begin{equation}\label{2.2}
(\nabla_{w,r,K}\cdot
(\mu\textbf{v}),\varphi)_K=-(\mu\textbf{v}_0,\nabla
\varphi)_K+\langle {\bv}_{b}\cdot\bn, \varphi\rangle_{\partial K},
\quad\forall\ \varphi\in P_r(K).
\end{equation}

Assume that $\bv_0$ is sufficiently smooth such that
$\nabla\cdot(\mu\bv_0)\in L^2(K)$. By applying the integration by
parts to the first term on the right-hand side of (\ref{2.2}), we
have
\begin{equation}\label{2.2new}
\begin{split}
(\nabla_{w,r,K} \cdot(\mu\bv),\varphi)_K =(\nabla\cdot(\mu\bv_0),
\varphi)_K+\langle (\bv_{b}-\mu\bv_0)\cdot\bn,
\varphi\rangle_{\partial K},
\end{split}
\end{equation}
for any $\varphi\in [P_r(K)]^3$.

\subsection{Weak curl}
The weak curl of $\textbf{v} \in V(K)$ (see \cite{mwyz}), denoted by
$\nabla_w \times \textbf{v}$, is defined as a bounded linear
functional in the Sobolev space $[H^1(K)]^3$, such that
$$
\langle \nabla_w \times \textbf{v},\varphi\rangle_K=(\textbf{v}_0,
\nabla\times \varphi)_K-\langle \textbf{v}_{b}\times
\textbf{n},\varphi\rangle_{\partial K},\quad\ \forall \ \varphi\in
[H^1(K)]^3.
$$

The discrete weak curl of $\textbf{v} \in V(K)$, denoted by
$\nabla_{w,r,K} \times \textbf{v}$, is defined as the unique
polynomial in $[P_r (K)]^3$,  satisfying
\begin{equation}\label{2.5}
(\nabla_{w,r,K} \times \textbf{v},\varphi)_K=(\textbf{v}_0,
\nabla\times \varphi)_K-\langle \textbf{v}_{b}\times
\textbf{n},\varphi\rangle_{\partial K}, \quad\forall\varphi\in [P_r
(K)]^3.
\end{equation}

For sufficiently smooth $\bv_0$ such that $\nabla\times\bv_0\in
[L^2(K)]^3$, by applying the integration by parts to the first term
on the right-hand side of (\ref{2.5}), we obtain
\begin{equation}\label{2.5new}
\begin{split}
(\nabla_{w,r,K} \times \bv,\varphi)_K =(\nabla\times\bv_0,
\varphi)_K-\langle (\bv_{b}-\bv_0)\times\bn,
\varphi\rangle_{\partial K},
\end{split}
\end{equation}
for any $\varphi\in [P_r(K)]^3$.

\section{Weak Galerkin Discretizations}

A polyhedral partition of $\Omega$ is a family of polyhedra $\{T_j:\
j=1,2,\ldots\}$ such that the two following conditions are
satisfied: (1) the union of all polyhedra $T_j$ is the domain
$\Omega$, and (2) the intersection of any two polyhedra, $T_j\cap
T_i \ (i\neq j)$, is either empty or a common face of $T_j$ and
$T_i$. Each partition cell $T_j$ is called an element. A polyhedral
partition with finite number of elements is called a finite element
partition of $\Omega$. Assume that $\{T_j\}_{j=1,\ldots, N}$ is a
finite element partition of $\Omega$ that is shape-regular according
to \cite{wy1202}. Denote by $h_T=\diameter(T)$ the diameter of the
cell/element $T$, and $h=\max_{T} h_T$ the meshsize of the partition
$\T_h=\{T_j\}_{j=1,\ldots, N}$. Denote by ${\cal E}_h$ the set of
all faces in ${\cal T}_h$ so that each $e\in \E_h$ is either on the
boundary of $\Omega$ or shared by two elements. Denote by
$\E^0_h=\E_h\setminus {\partial\Omega}$ the set of all interior
faces in $\E_h$. By definition, for each interior face $e\in
\E_h^0$, there are two elements $T_j$ and $T_i$, $i\neq j$, such
that $e=T_j\cap T_i$.

Let $k\ge 1$ be any integer. For each element $T\in\T_h$, define the
local finite element space as
$$
\bV(k,T)=\{ \bv=\{\bv_0,\bv_b\}: \bv_0\in [P_k(T)]^3, \bv_b\in
[P_{k}(e)]^3,\ e\in (\partial T\cap\E_h)\}.
$$
The global weak finite element space for the vector-component is
given by
\begin{equation}\label{EQ:global-WFES}
\bV_h=\{\bv=\{\bv_0,\bv_b\}:\; \bv|_T\in \bV(k,T),\
\bv_b|_{\pT_1\cap e}=\bv_b|_{\pT_2\cap e},\ T\in \T_h, e\in\E_h^0\},
\end{equation}
where $\bv_b|_{\pT_j\cap e}$ is the value of $\bv_b$ on the face $e$
as seen from the element $T_j,\ j=1,2$. The finite element space for
the Lagrange multiplier $p$ is defined as
$$
W_h =\{q: \ q\in L^2(\Omega), \ q|_T\in P_{k-1}(T), T\in\T_h\}.
$$

The discrete weak divergence $(\nabla_{w,k-1}\cdot ) $ and the
discrete weak curl $(\nabla_{w,k-1} \times)$ can be computed by
using (\ref{2.2}) and (\ref{2.5}) on each element; i.e.,
\begin{align*}
(\nabla_{w,k-1}\cdot (\mu\textbf{v}))|_T=&\nabla_{w,k-1,T}\cdot
(\mu\textbf{v}|_T), \quad \ \textbf{v}\in \bV_h,\\
(\nabla_{w,k-1}\times  \textbf{v})|_T=&\nabla_{w,k-1,T}\times
(\textbf{v}|_T), \quad \textbf{v}\in \bV_h.
\end{align*}
For simplicity of notation, we shall drop the subscript $k-1$ from
the notations $(\nabla_{w,k-1}\cdot )$ and
$(\nabla_{w,k-1}\times)$ from now on.

Introduce the following bilinear forms
\begin{align}\label{EQ:November20-2014:001}
a(\textbf{v},\textbf{w})=&(\kappa\nabla_{w}\times
\textbf{v},\nabla_{w} \times
\textbf{w})_h+s(\textbf{v},\textbf{w}),\\
b(\textbf{v},q)=&(\nabla_{w}\cdot (\mu\textbf{v}),
q)_h,\label{EQ:November20-2014:002}
\end{align}
where
\begin{align}\nonumber
(\kappa\nabla_{w}\times \textbf{v},\nabla_{w} \times \textbf{w})_h
=&\sum_{T\in {\cal T}_h}(\kappa\nabla_{w}\times \textbf{v},\nabla_{w} \times \textbf{w})_T,\\
(\nabla_{w}\cdot (\mu\textbf{v}), q)_h =&\sum_{T\in {\cal
T}_h}(\nabla_{w}\cdot (\mu\textbf{v}), q)_T,
\nonumber \\
 s(\textbf{v},\textbf{w})=&\sum_{T\in {\cal
T}_h}h_T^{-1}\langle
 (\mu\textbf{v}_0-\textbf{v}_b)\cdot\bn,
(\mu\textbf{w}_0-\textbf{w}_b)\cdot\bn\rangle_{\partial T}\label{EQ:StabilityTerm}\\
&+\sum_{T\in {\cal T}_h}h_T^{-1}\langle
 (\textbf{v}_0-\textbf{v}_b)\times\bn,
(\textbf{w}_0-\textbf{w}_b)\times\bn\rangle_{\partial T}.\nonumber
\end{align}

\bigskip

We are now in a position to describe a finite element method for the
model problem given in (\ref{EQ:model-problem-I}) and
(\ref{EQ:model-problem-I-new}). Consider first the variational
formulation (\ref{EQ:model-problem-I-new}). Observe that the test
space is $H_0({\rm curl}; \Omega)\cap H({\rm div}_\mu;\Omega)$. The
corresponding analogue in the weak Galerkin setting is the following
weak finite element space
\begin{equation}\label{EQ:global-WFES-testspace-problem1}
\bV_h^{0}=\{\bv=\{\bv_0,\bv_b\}\in\bV_h:\; \bv_b\times\bn = 0 \
\mbox{on} \ \Gamma\}.
\end{equation}

\begin{algorithm}\label{algo1} \emph{(weak Galerkin for the problem
(\ref{EQ:model-problem-I-new}))} Find $\textbf{u}_h= \{\textbf{u}_0,
\textbf{u}_b\} \in \bV_h$, $p_h\in W_h$, and $\lambda\in \bbR^m$
with $\bu_b\times\bn = \bQ_b\xi$ on $\Gamma$ such that
\begin{align}\label{3.3}
a(\textbf{u}_h,\textbf{v})+b(\textbf{v},p_h)+\sum_{i=1}^m\langle
\bv_b\cdot\bn_i, \lambda_i\rangle_{\Gamma_i}= (\bg,\textbf{v}_0),
\qquad \forall \ \textbf{v}=\{\textbf{v}_0, \textbf{v}_b\}\in
\bV_h^{0},
\\
b(\textbf{u}_h,w)+\sum_{i=1}^m\langle \bu_b\cdot\bn_i,
s_i\rangle_{\Gamma_i} =(f,w)+\sum_{i=1}^m\beta_i s_i, \qquad \forall
\ w\in W_h, \ s\in\bbR^m.\label{3.4}
\end{align}
Here $\bQ_b\xi$ is the usual $L^2$-projection to $[P_{k}(e)]^3$ for
each boundary face $e\in (\Gamma\cap\E_h)$.
\end{algorithm}

The Lagrange multiplier $\lambda$ can be eliminated from the
formulation (\ref{3.3})-(\ref{3.4}) if the test space $\bV_h^{0}$ is
replaced by
\begin{equation}\label{EQ:global-WFES-testspaceU-problem1}
\bU_h^{0}=\{\bv=\{\bv_0,\bv_b\}\in\bV_h^{0}:\; \langle
\bv_b\cdot\bn_i, 1\rangle_{\Gamma_i}=0,\ i=1,\ldots, m\}.
\end{equation}
The following is the corresponding weak Galerkin finite element
scheme.

\begin{algorithm}\label{algo1-original} \emph{(weak Galerkin for the
problem (\ref{EQ:model-problem-I}))} Find $\textbf{u}_h=
\{\textbf{u}_0, \textbf{u}_b\} \in \bV_h$ and $p_h\in W_h$ with
$\bu_b\times\bn = \bQ_b\xi$ on $\Gamma$ such that
\begin{align}\label{3.3-original}
a(\textbf{u}_h,\textbf{v})+b(\textbf{v},p_h) & = (\bg,\textbf{v}_0),
\qquad \forall \ \textbf{v}=\{\textbf{v}_0, \textbf{v}_b\}\in
\bU_h^{0},
\\
b(\textbf{u}_h,w) & = (f,w), \qquad \forall \ w\in W_h, \label{3.4-original}\\
\langle \bu_b\cdot\bn_i, 1\rangle_{\Gamma_i} & =
\beta_i,\qquad\quad\ \  i=1,\ldots, m.\label{3.5-original}
\end{align}
\end{algorithm}

Note that the Algorithms \ref{algo1} and \ref{algo1-original} are
equivalent in the sense that both give the same numerical solution
$\bu_h$ and $p_h$. Thus, it is sufficient to develop a convergence
theory for Algorithm \ref{algo1-original} only.

\section{Existence and Uniqueness}

The goal of this section is to show that the WG Algorithm
\ref{algo1-original} has one and only one solution.

We first introduce a topology in the weak finite element space
$\bV_h$ by defining a semi-norm as follows
\begin{equation}\label{Eq:bar1-norm}
\begin{split}
\3bar\textbf{v}\3bar_1 = &\Big(\sum_{T\in {\cal T}_h}
\|\nabla_{w}\times \textbf{v}\|_T^2+ \| \nabla_w\cdot
(\mu\textbf{v}) \|^2_{ T} \\
& + h_T^{-1}\|
 \mu\textbf{v}_0\cdot\bn-\textbf{v}_b\cdot\bn\|^2_{\partial T}
+ h_T^{-1}\|
 \textbf{v}_0\times\bn-\textbf{v}_b\times\bn\|^2_{\partial T}
 \Big)^{\frac{1}{2}}.
 \end{split}
\end{equation}
For convenience, we set
\begin{equation}\label{Eq:bar-norm}
\begin{split}
\3bar\textbf{v}\3bar^2 := a(\bv, \bv)= &\sum_{T\in {\cal T}_h}
(\kappa\nabla_{w}\times \textbf{v}, \nabla_w\times\bv)_T\\
&+ h_T^{-1}\|
 \mu\textbf{v}_0\cdot\bn-\textbf{v}_b\cdot\bn\|^2_{\partial T}
+ h_T^{-1}\|
 \textbf{v}_0\times\bn-\textbf{v}_b\times\bn\|^2_{\partial T}
 \Big)^{\frac{1}{2}}.
 \end{split}
\end{equation}

In the finite element space $W_h$, we introduce a mesh-dependent
norm
\begin{equation}\label{Eq:bar-norm-Wh0}
\|q\|_{W_h}^2 = h^2 \sum_{T\in {\cal T}_h}\|\nabla q\|_T^2 +
h\sum_{e\in {\cal E}_h^0} \|\jump{q}_e\|_e^2 + h\sum_{i=0}^m
\|q-\bar q_i\|_{\Gamma_i}^2,
\end{equation}
where $\jump{q}_e$ stands for the jump of $q$ on the interior face
$e\in \E_h^0$, $\bar q_0=0$, and $\bar q_i$ is the average of $q$ on
the connected boundary component $\Gamma_i,\ i=1,\ldots, m$.

\begin{lemma}\label{Lemma:3bar-one-is-a-norm}
Assume that the domain $\Omega$ is connected. Then, the semi-norm
$\3bar\cdot\3bar_1$ defined as in (\ref{Eq:bar1-norm}) defines a
norm in the linear space $\bU_h^{0}$.
\end{lemma}

\begin{proof} It suffices to verify the
positivity property for $\3bar\cdot\3bar_1$. To this end, assume
that $\3bar\bv\3bar_1=0$ for some $\bv\in \bU_h^0$. Thus,
\begin{eqnarray}\label{Eq:Feb8:800}
\nabla_w\times\bv &=& 0,\qquad \mbox{in}\
T,\\
\quad \nabla_w\cdot(\mu\bv) &=& 0, \qquad \mbox{in}\
T,\label{Eq:Feb8:801}\\
 \mu\bv_0\cdot\bn - \bv_b\cdot\bn &=& 0,\qquad \mbox{on}\
\partial T\label{Eq:Feb8:802} \\
\bv_0\times\bn - \bv_b\times\bn &=& 0,\qquad \mbox{on}\
\partial T.\label{Eq:Feb8:8903}
\end{eqnarray}

Using (\ref{Eq:Feb8:800}), (\ref{Eq:Feb8:8903}) and (\ref{2.5new}),
for any $\varphi\in [P_{k-1}(T)]^3$, we have
\begin{equation*}
\begin{split}
0=&(\nabla_{w}\times \bv,\varphi)_T\\
=& (\nabla\times \bv_0,\varphi)_T-\langle( \bv_b-\bv_0)\times
\bn,\varphi\rangle_{\partial T} \\
=& (\nabla\times \bv_0,\varphi)_T.
\end{split}
\end{equation*}
It follows that $\nabla\times \bv_0=0$ on each element $T\in {\cal
T}_h$, which along with (\ref{Eq:Feb8:8903}) implies $\bv_0\in
H({\rm curl};\Omega)$ and
\begin{equation}\label{Eq:Feb8:14803}
\nabla\times\bv_0=0,\qquad \mbox{in} \ \Omega.
\end{equation}

Next, using (\ref{Eq:Feb8:801}), (\ref{Eq:Feb8:802}) and
(\ref{2.2new}), for any $\varphi\in P_{k-1}(T)$, we have
\begin{equation*}
\begin{split}
0=&(\nabla_{w}\cdot(\mu\bv),\varphi)_T   \\
=& (\nabla\cdot(\mu\bv_0),\varphi)_T+\langle( \bv_b-\mu\bv_0)\cdot
\bn,\varphi\rangle_{\partial T}  \\
=& (\nabla\cdot(\mu\bv_0),\varphi)_T.
\end{split}
\end{equation*}
It follows that $\nabla\cdot(\mu\bv_0)=0$ on each element
$T\in\T_h$, which, together with (\ref{Eq:Feb8:802}), gives rise to
$\mu\bv_0\in H({\rm div};\Omega)$ and
\begin{equation}\label{Eq:Feb8:1803}
\nabla\cdot(\mu\bv_0)=0,\qquad \mbox{in} \ \Omega.
\end{equation}

Combining (\ref{Eq:Feb8:802})-(\ref{Eq:Feb8:8903}) with the fact
that $\bv\in \bU_h^{0}$ yields
\begin{equation*}
\begin{split}
\bv_0\times\bn & = 0, \quad  \mbox{on}\ \Gamma,\\
\langle\mu\bv_0\cdot\bn_i, 1\rangle_{\Gamma_i} & =0,\quad
i=1,\ldots, m,
\end{split}
\end{equation*}
which, along with (\ref{Eq:Feb8:14803}) and (\ref{Eq:Feb8:1803}),
implies $\bv_0=0$ and hence $\bv_b=0$; see the proof of Lemma
\ref{Lemma:Solution4P1IsUnique} for details.
\end{proof}

\begin{theorem}\label{theorem3.1} The weak Galerkin finite element algorithms
\ref{algo1} and \ref{algo1-original} have one and only one solution.
\end{theorem}

\begin{proof}
Since the Algorithms \ref{algo1} and \ref{algo1-original} are
equivalent, then it suffices to deal with Algorithms
\ref{algo1-original}. As the number of unknowns is the same as the
number of equations in (\ref{3.3-original})-(\ref{3.5-original}),
the existence of solution is equivalent to the uniqueness.

Let $(\bu_h^{(j)}; p_h^{(j)}) \in \bV_h\times W_h, \ j=1,2,$ be two
solutions of (\ref{3.3-original})-(\ref{3.5-original}), and set
$$
\bz_h =\bu_h^{(1)} - \bu_h^{(2)},\quad \gamma_h =p_h^{(1)} -
p_h^{(2)}.
$$
It is clear that $(\bz_h;\gamma_h)\in \bU_h^{0}\times W_h$ satisfies
\begin{align}\label{Nov.27.001}
a(\textbf{z}_h,\textbf{v})+b(\textbf{v},\gamma_h) & = 0, \qquad
\forall \ \textbf{v}=\{\textbf{v}_0, \textbf{v}_b\}\in \bU_h^{0},
\\
b(\textbf{z}_h,w) & = 0, \qquad \forall \ w\in W_h.
\label{Nov.27.002}
\end{align}
By letting $\bv=\bz_h$ in (\ref{Nov.27.001}) and $w=\gamma_h$ in
(\ref{Nov.27.002}), we obtain
$$
a(\bz_h, \bz_h) = 0,\quad \nabla_w \cdot(\mu\bz_h) = 0,
$$
which leads to $\3bar\bz_h\3bar_1=0$. It follows from Lemma
\ref{Lemma:3bar-one-is-a-norm} that $\bz_h =0$, and thus
$\bu_h^{(1)}\equiv \bu_h^{(2)}$.

To show $\gamma_h=0$, we use (\ref{Nov.27.001}) and $\bz_h=0$ to
obtain
$$
b(\textbf{v},\gamma_h) = 0, \qquad \forall \
\textbf{v}=\{\textbf{v}_0, \textbf{v}_b\}\in \bU_h^{0}.
$$
From the definition of weak divergence (\ref{2.2}),
\begin{equation}\label{Nov.10}
\begin{split}
b(\bv,\gamma_h) & = \sum_{T\in {\cal T}_h} (\nabla_w \cdot (\mu\bv),\gamma_h)_T\\
& = \sum_{T\in {\cal T}_h} -(\mu\bv_0,\nabla \gamma_h)_T+\langle
\bv_b\cdot \bn, \gamma_h\rangle_{\partial T}\\
& = - \sum_{T\in {\cal T}_h} (\mu\bv_0,\nabla \gamma_h)_T +
\sum_{e\in\E_h} \langle \bv_b\cdot\bn_e, \jump{\gamma_h}\rangle_e,
\end{split}
\end{equation}
where $\bn_e$ is a prescribed orientation of $e$. By letting
$\bv=\{-h^2\nabla \gamma_h;\  h \delta_e\jump{\gamma_h}\bn_e\}$ with
$\delta_e=0$ when $e\subset \Gamma_i,\ i=1,\ldots, m$ and
$\delta_e=1$ otherwise, we see that $\bv\in \bU_h^{0}$. Substituting
this into (\ref{Nov.10}) yields
\begin{equation*}
0=b(\bv,\gamma_h)=h^2\sum_{T\in {\cal T}_h} (\mu\nabla \gamma_h,
\nabla\gamma_h)_T + h\sum_{e\in \E_h^0\cup \Gamma_0}
\|\jump{\gamma_h}\|_e^2.
\end{equation*}
It follows that $\gamma_h=0$, and thus $p_h^{(1)}\equiv p_h^{(2)}$.
This completes the proof of uniqueness.
\end{proof}

\section{Error Equations}

Let $Q_0$ be the $L^2$ projection onto $[P_k(T)]^3, \ T\in {\cal
T}_h$, and $Q_b$ the $L^2$ projection onto $[P_{k}(e)]^3, \ e\in
\partial T\cap \E_h$. Denote by $Q_h$ the $L^2$ projection onto the
weak finite element space $\bV_h$ such that on each element $T\in
{\cal T}_h$,
\begin{equation}\label{EQ:12-31:001}
(Q_h \textbf{u})|_T= \{Q_{0}\textbf{u}, \mathds{Q}_{b} \textbf{u}\},
\end{equation}
where
\begin{equation}\label{EQ:12-31:002}
\mathds{Q}_{b} \textbf{u} = Q_b (\mu\bu\cdot\bn)\bn + Q_b
(\bn\times(\bu\times\bn)).
\end{equation}
Note that $\bn\times(\bu\times\bn)=\bu-(\bu\cdot\bn)\bn$ is the
tangential component of the vector $\bu$ on the boundary of the
element. When $\mu=I$ is the identity matrix, $(\mu\bu\cdot\bn)\bn$
is the normal component of $\bu$. In general, $(\mu\bu\cdot\bn)\bn
+\bn\times(\bu\times\bn)$ is not a decomposition of the vector $\bu$
restricted on $\partial T$.

Denote by ${\cal Q}_h$ and $\textbf{Q}_h$ the $L^2$ projections onto
$P_{k-1}(T)$ and $[P_{k-1}(T)]^{3}$, respectively.

\begin{lemma}\cite{mwyz, wy1302} \label{lemma4.2} The
projection operators $Q_h$, $\textbf{Q}_h$, and ${\cal Q}_h$ satisfy
the following commutative identities:
\begin{align}\label{4.4}
\nabla_w \cdot(\mu Q_h \textbf{v})=&{\cal Q}_h \nabla \cdot
(\mu\textbf{v}), \qquad \textbf{v}\in H({\rm div}_\mu;\Omega),\\
\nabla_w \times (Q_h \textbf{v})=& \textbf{Q}_h (\nabla \times
\textbf{v}), \qquad \textbf{v}\in H({\rm curl};\Omega).\label{4.5}
\end{align}
\end{lemma}

\begin{proof}
It suffices to verify (\ref{4.4}) on each element $T\in\T_h$. To
this end, using the definition (\ref{2.2}) for the discrete weak
divergence and (\ref{EQ:12-31:002}), we obtain
\begin{eqnarray*}
(\nabla_w \cdot(\mu Q_h \textbf{v}), \varphi)_T &=& - (\mu Q_{0}\bv,
\nabla\varphi)_T + \langle (\mathds{Q}_b\bv)\cdot\bn,
\varphi\rangle_\pT\\
&=& - (\mu Q_{0}\bv, \nabla\varphi)_T + \langle
{Q}_b(\mu\bv\cdot\bn),\varphi\rangle_\pT \\
&=& - (\mu \bv, \nabla\varphi)_T + \langle
\mu\bv\cdot\bn,\varphi\rangle_\pT \\
&=&(\nabla\cdot(\mu\bv), \varphi)_T \\
&=&({\cal Q}_h\nabla\cdot(\mu\bv), \varphi)_T
\end{eqnarray*}
for all $\varphi\in P_{k-1}(T)$. Thus, the identity (\ref{4.4})
holds true. A similar argument can be applied to verify (\ref{4.5}).
\end{proof}

Let $(\bu_h;p_h)=( \{\bu_0, \bu_b\};p_h) \in \bV_h\times W_h$ be the
WG finite element solution arising from Algorithm
\ref{algo1-original}, and $(\bu;p)$ be the solution of the
continuous problem (\ref{EQ:model-problem-I}) or
(\ref{EQ:model-problem-I-new}). The error functions are given by
\begin{align}\label{5.1}
\textbf{e}_h&=\{\textbf{e}_0,\textbf{e}_b\}=\{Q_0\textbf{u}-\textbf{u}_0,
Q_b\textbf{u}-\textbf{u}_b\},\\
 \epsilon_h&={\cal Q}_h p-p_h.\label{ph}
\end{align}

\begin{lemma}\label{lemma5.1}
Assume that $(\bw; \rho)\in  H({\rm curl};\Omega) \times
L^2(\Omega)$ is sufficiently smooth on each element $T\in \T_h$
satisfying
\begin{eqnarray}
&\nabla\times(\kappa\nabla \times \bw) - \mu\nabla \rho =
\eta,\qquad
 \mbox{in}\ \Omega, \label{5.2}\\
&\rho|_{\Gamma_0}=0,\quad \rho|_{\Gamma_i}=const,\ i=1,\ldots, m.
\label{November.28.100}
\end{eqnarray}
Denote by ${\cal Q}_h\rho$ the $L^2$ projection of $\rho$ in the
finite element space $W_h$. Then,
\begin{equation}\label{Div-Curl:Feb9:300-new}
(\kappa\nabla_w \times (Q_h\bw),\nabla_w \times\textbf{v}
)_h+(\nabla_w\cdot (\mu\textbf{v}),{\cal
Q}_h\rho)_h=(\eta,\textbf{v}_0)+l_\bw
(\textbf{v})+\theta_\rho(\textbf{v}),
\end{equation}
for all $\textbf{v}\in \bU_h^{0}$. Here $l_\bw(\textbf{v})$ and
$\theta_\rho(\textbf{v})$ are two functionals in the linear space
$\bV_h$ given by
\begin{align}\label{l}
l_\bw(\textbf{v})&= \sum_{T\in{\cal T}_h}\langle
(\textbf{Q}_h-I)(\kappa\nabla \times\bw),
(\textbf{v}_0-\textbf{v}_b)\times \textbf{n}\rangle_{\partial T},\\
\theta_\rho(\textbf{v})&=\sum_{T\in{\cal T}_h}\langle \rho- {\cal
Q}_h\rho,
 (\mu\textbf{v}_0-\textbf{v}_b)\cdot\bn\rangle_{\partial T}.\label{theta}
\end{align}
\end{lemma}

\begin{proof}
It follows from (\ref{2.5new}) with $\varphi= \kappa\nabla_w \times
(Q_h\textbf{w})$ that
\begin{equation}\label{Eq:Feb8:500}\nonumber
\begin{split}
 (\nabla_w \times & \textbf{v},  \kappa\nabla_w \times
(Q_h\textbf{w}))_T = \\
 & \ (\nabla \times \textbf{v}_0,  \kappa\nabla_w \times (Q_h\textbf{w}))
_T -\langle (\textbf{v}_b-\textbf{v}_0)\times\bn,
 \kappa\nabla_w \times (Q_h\textbf{w})\rangle_{\partial T}.
\end{split}
\end{equation}
Using (\ref{4.5}), the above equation can be rewritten as
\begin{equation}\label{5.4}\nonumber
\begin{split}
( \kappa\nabla_w \times & (Q_h\textbf{w}),\nabla_w \times \textbf{v}
)_T = \\
  & ( \kappa\nabla\times \bw, \nabla \times \bv_0)_T+\langle
 \textbf{Q}_h (\kappa\nabla\times \bw), (\bv_0-\bv_b)\times \bn
\rangle_{\partial T}.
\end{split}
\end{equation}
Applying the integration by parts to the first term on the
right-hand side yields
\begin{equation}\label{5.4.001}
\begin{split}
( \kappa\nabla_w \times & (Q_h\textbf{w}),\nabla_w \times\textbf{v}
)_T \\
 = & (\curl(\kappa\curl\bw), \bv_0)_T  - \langle \kappa\curl \bw, \bv_0\times\bn\rangle_\pT\\
 & +\langle \textbf{Q}_h (\kappa\nabla\times \bw), (\bv_0-\bv_b)\times \bn
\rangle_{\partial T}\\
= & (\curl(\kappa\curl\bw), \bv_0)_T - \langle \kappa\curl \bw,
\bv_b\times\bn\rangle_\pT\\
& + \langle (\textbf{Q}_h-I) (\kappa\nabla\times \bw),
(\bv_0-\bv_b)\times \bn \rangle_{\partial T}.
\end{split}
\end{equation}

Using (\ref{2.2new}) with $\varphi={\cal Q}_h \rho$ and the usual
integration by parts, we obtain
\begin{equation}\label{Eq:Feb8:501}
\begin{split}
&\ (\nabla_w\cdot(\mu\bv),{\cal Q}_h \rho)_T\\
  = &\ (\nabla \cdot(\mu\bv_0), {\cal Q}_h \rho)_T+\langle
(\bv_b-\mu\bv_0)\cdot\bn,{\cal Q}_h
\rho\rangle_{\partial T}\\
 =&\ (\nabla \cdot (\mu\bv_0), \rho)_T+\langle
(\bv_b-\mu\bv_0)\cdot\bn,{\cal
Q}_h\rho\rangle_{\partial T}\\
  =&\ -(\mu\bv_0, \nabla \rho)_T+\langle \mu\bv_0\cdot\bn,
\rho\rangle_{\partial T} +\langle (\bv_b-\mu\bv_0)\cdot\bn,{\cal
Q}_h
\rho\rangle_{\partial T}\\
 =&\ -(\bv_0, \mu\nabla \rho)_T+\langle
(\bv_b-\mu\bv_0)\cdot\bn,{\cal Q}_h \rho-\rho\rangle_{\partial T}
+\langle \bv_b\cdot\bn, \rho\rangle_{\partial T}.
\end{split}
\end{equation}

Summing (\ref{5.4.001}) over all the elements $T\in\T_h$ yields
\begin{equation}\label{5.4.002}
\begin{split}
( \kappa\nabla_w \times & (Q_h\textbf{w}),\nabla_w \times \textbf{v}
)_h = \sum_{T\in\T_h} (\curl(\kappa\curl\bw), \bv_0)_T  \\
  & +\sum_{T\in\T_h}\langle
 (\textbf{Q}_h-I) (\kappa\nabla\times \bw), (\bv_0-\bv_b)\times \bn
\rangle_{\partial T},
\end{split}
\end{equation}
where we have used two properties: (1) the cancelation property for
the boundary integrals on interior faces, and (2) the fact that
$\bv_b\times\bn=0$ on $\Gamma$. Similarly, summing
(\ref{Eq:Feb8:501}) over all the elements $T\in\T_h$, we obtain
\begin{equation}\label{5.4.003}
\begin{split}
(\nabla_w\cdot (\mu\textbf{v}),{\cal Q}_h \rho)_h = &
-(\bv_0,\mu\nabla \rho) + \sum_{T\in {\cal T}_h} \langle
(\mu\bv_0-\bv_b)\cdot\bn,
\rho-{\cal Q}_h \rho \rangle_{\partial T}\\
& + \sum_{e\in \E_h\cap \Gamma} \langle \bv_b\cdot\bn,
\rho\rangle_e.
\end{split}
\end{equation}
The third term on the right-hand side of (\ref{5.4.003}) vanishes if
$\bv\in \bU_h^{0}$ and $\rho$ satisfies the boundary condition
(\ref{November.28.100}). Thus, the equation
(\ref{Div-Curl:Feb9:300-new}) holds true by adding up
(\ref{5.4.002}) and (\ref{5.4.003}). This completes the proof of the
lemma.
\end{proof}

\begin{theorem} \label{Thm:div-curl:theorem-error-eqns}
Let $(\bu; p)$ be the solution of the model problem
(\ref{EQ:model-problem-I}) or (\ref{EQ:model-problem-I-new}) and
$(\bu_h; p_h)$ be its numerical solution arising from the WG finite
element scheme (\ref{3.3-original})-(\ref{3.5-original}). Let the
error functions $\textbf{e}_h$ and $\epsilon_h$ be defined by
(\ref{5.1})-(\ref{ph}). Then, $\be_h\in\bU_h^{0}$ and the following
error equations hold true
\begin{align}\label{EQ:div-curl:error-eq-01}
a(\textbf{e}_h,\textbf{v})+b(\textbf{v},\epsilon_h)&=
\varphi_{\textbf{u},p}(\textbf{v}), \qquad \forall \textbf{v}\in \bU_h^{0},\\
b(\textbf{e}_h,q)&=0,\qquad \qquad \ \forall q\in
W_h,\label{EQ:div-curl:error-eq-02}
\end{align}
where
\begin{equation}\label{EQ:div-cul:varphi-up}
\varphi_{\textbf{u},p}(\textbf{v})=l_\textbf{u}(\textbf{v})+
\theta_p(\textbf{v})+ s(Q_h\textbf{u},\textbf{v}).
\end{equation}
\end{theorem}

\begin{proof}
Consider the model problem (\ref{EQ:model-problem-I-new}), as
(\ref{EQ:model-problem-I}) is equivalent to
(\ref{EQ:model-problem-I-new}). Let $(\bu; p; \lambda)$ be the
solution of this model problem. It is not hard to see that the
following holds true:
\begin{equation}\nonumber\label{5.2.Nov-28.200}
\begin{split}
 &\nabla\times( \kappa \nabla \times \bu) - \mu \nabla p  = \bg,\qquad
 \mbox{in}\ \Omega,\\
& p|_{\Gamma_0} = 0,\ p|_{\Gamma_i} =-\lambda_i,\quad i=1,\ldots, m.
\end{split}
\end{equation}
Thus, by Lemma \ref{lemma5.1}, we have
$$
(\kappa\nabla_w\times(Q_h\textbf{u}),\nabla_w\times
\textbf{v})_h+(\nabla_w \cdot (\mu\textbf{v}),{\cal
Q}_hp)_h=(\textbf{g},\textbf{v}_0)
+l_\textbf{u}(\textbf{v})+\theta_p(\textbf{v})
$$
for all $\bv\in \bU_h^{0}$. It follows that
\begin{equation}\label{4.13}
a(Q_h\textbf{u},\textbf{v})+b(\textbf{v},{\cal
Q}_hp)=(\textbf{g},\textbf{v}_0)+l_\textbf{u}(\textbf{v})+
\theta_p(\textbf{v})+s(Q_h\textbf{u},\textbf{v}).
\end{equation}
Subtracting (\ref{3.3-original}) from (\ref{4.13}) gives the first
error equation (\ref{EQ:div-curl:error-eq-01}).

Next, from the second equation in (\ref{EQ:model-problem-I}) and the
commutative relation (\ref{4.4}),
 \begin{equation}\label{4.14}
(f,q)=(\nabla\cdot (\mu\textbf{u}),q)_h=({\cal Q}_h\nabla\cdot
(\mu\textbf{u}),q)_h=(\nabla_w\cdot(\mu Q_h\textbf{u}),q)_h.
\end{equation}
The difference of (\ref{4.14}) and (\ref{3.4-original}) yields the
second error equation (\ref{EQ:div-curl:error-eq-02}). This
completes the proof.
\end{proof}

\section{The {\em inf-sup} Condition}
For any $q\in W_h$, define a finite element function
$\textbf{v}_q\in \bU_h^{0}$ by $\bv_q=\{-h^2\nabla q;\ h
\bv_{q,b}\}$:
\begin{equation}\label{November-30:500}
\bv_{q,b}=\left\{
\begin{array}{ll}
\jump{q}\ \bn_e,&\qquad \mbox{ on } \ e\in\E_h^0,\\
(q-\bar q_i)\ \bn_i,&\qquad \mbox{ on } \ e\in \E_h\cap \Gamma_i, \
i=0,1,\ldots, m.
\end{array}
\right.
\end{equation}
Recall that $\jump{q}$ is the jump of $q$ on the corresponding face
$e\in\E_h^0$, $\bn_e$ is a prescribed orientation of $e$, $\bn_i$ is
the outward normal direction on the connected component $\Gamma_i$,
$\bar q_0=0$, and $\bar q_i$ is the average of $q$ on $\Gamma_i, \
i=1,\ldots, m$.

For any $\bv=\{\bv_0; \bv_b\}\in \bU_h^0$, from the definition of
weak divergence (\ref{2.2}), we have
\begin{equation*}\label{November.30.001}
\begin{split}
b(\bv,q) & = \sum_{T\in {\cal T}_h} (\nabla_w \cdot (\mu\bv),q)_T\\
& = \sum_{T\in {\cal T}_h} -(\mu\bv_0,\nabla q)_T+\langle \bv_b
\cdot \bn, q\rangle_{\partial T}\\
& = - \sum_{T\in {\cal T}_h} (\mu\bv_0,\nabla q)_T +
\sum_{e\in\E_h^0} \langle \bv_b\cdot\bn_e, \jump{q}\rangle_e
+\sum_{i=0}^m \langle \bv_b\cdot\bn_i, q\rangle_{\Gamma_i}.
\end{split}
\end{equation*}
Note that $\langle \bv_b\cdot\bn_i, 1\rangle_{\Gamma_i}=0$ for $
i=1,\ldots,m$. Thus,
\begin{equation}\label{November.30.001-new}
b(\bv,q)=-\sum_{T\in {\cal T}_h} (\mu\bv_0,\nabla q)_T +
\sum_{e\in\E_h^0} \langle \bv_b\cdot\bn_e, \jump{q}\rangle_e
+\sum_{i=0}^m \langle \bv_b\cdot\bn_i, q-\bar q_i\rangle_{\Gamma_i}.
\end{equation}

\begin{lemma} \emph{(inf-sup condition)} \label{lvq2}
For any $q\in W_h$, there exists a finite element function
$\textbf{v}_q\in \bU_h^{0}$ such that
\begin{eqnarray}\label{November:29:801}
b(\bv_q,q) &=& h^2\sum_{T\in {\cal T}_h} (\mu\nabla q,\nabla q)_T +
h\sum_{e\in \E_h^0} \|\jump{q}\|_e^2+h\sum_{i=0}^m\|q-\bar q_i\|_{\Gamma_i}^2, \\
\label{6.5} \3bar\bv_q\3bar &\leqC &\|q\|_{W_h}.
\end{eqnarray}
\end{lemma}

\begin{proof} For any $q\in W_h$, define $\bv_{q,b}$ by (\ref{November-30:500}) and set
$\textbf{v}_q = \{-h^2\nabla q;\ h \bv_{q,b}\}$. On the boundary
$\Gamma$, the vector $\bv_{q,b}$ is parallel to the normal director
$\bn$. Thus, we have $\bv_{q,b}\times\bn=0$ on $\Gamma$. Moreover,
on each connected component $\Gamma_i$, we have
$$
\langle \bv_{q,b}\cdot\bn_i, 1\rangle_{\Gamma_i} = \int_{\Gamma_i}
(q-\bar q_i) =0,\ i=1, 2,\ldots, m.
$$
Thus, $\bv_q\in\bU_h^{0}$.

Now, by taking $\bv=\bv_q$ in (\ref{November.30.001-new}), we obtain
\begin{equation*}
b(\bv_q,q)=h^2\sum_{T\in {\cal T}_h} (\mu\nabla q,\nabla q)_T +
h\sum_{e\in \E_h^0} \|\jump{q}\|_e^2+h\sum_{i=0}^m\|q-\bar
q_i\|_{\Gamma_i}^2,
\end{equation*}
which verifies the identity (\ref{November:29:801}).

To derive (\ref{6.5}), we consider the following decomposition
$$
\bv_q = \bv_q^{(1)} + \bv_q^{(2)},
$$
where $\bv_q^{(1)} = -\{h^2\nabla q;\ 0\}$ and $\bv_q^{(2)} = \{0;\
h\bv_{q,b}\}$. It suffices to establish (\ref{6.5}) for
$\bv_q^{(1)}$ and $\bv_q^{(2)}$ independently.

From the semi-norm definition (\ref{Eq:bar-norm}), we have
\begin{equation}\label{vq2}
\begin{split}
\3bar \textbf{v}_q^{(1)} \3bar^2 = & \sum_{T\in {\cal T}_h}
(\kappa\nabla_w \times \textbf{v}_q^{(1)}, \nabla_w \times
\textbf{v}_q^{(1)})_T\\
& + h_T^{-1}\|h^2 \mu\nabla q\cdot\bn \|^2_{\partial T} +
h_T^{-1}\|h^2\nabla q \times\bn\|^2_\pT.
\end{split}
\end{equation}
The definition (\ref{2.5}) for the discrete weak curl implies
$$
(\nabla_w \times \textbf{v}_q^{(1)}, \varphi)_T = -h^2 (\nabla q,
\curl\varphi)_T,\qquad \forall \ \varphi\in [P_{k-1}(T)]^3.
$$
It follows from the inverse inequality that
$$
\|\nabla_w \times \textbf{v}_q^{(1)}\|_T \leqC h \|\nabla q\|_T.
$$
Substituting the above into (\ref{vq2}) and then using the trace
inequality (\ref{Aa-trace}) yields
$$
\3bar \textbf{v}_q^{(1)} \3bar^2 \leqC h^2\|\nabla q\|_T^2,
$$
which verifies the estimate (\ref{6.5}) for $\textbf{v}_q^{(1)}$.

For $\textbf{v}_q^{(2)}$, we again use the semi-norm definition
(\ref{Eq:bar-norm}) to obtain
\begin{equation}\label{vq2-new}
\3bar \textbf{v}_q^{(2)} \3bar^2= \sum_{T\in {\cal T}_h}
(\kappa\nabla_w \times \textbf{v}_q^{(2)}, \nabla_w \times
\textbf{v}_q^{(2)})_T + h_T^{-1}\|h \bv_{q,b}\cdot\bn \|^2_{\partial
T} + h_T^{-1}\|h \bv_{q,b}\times\bn \|^2_{\pT}.
\end{equation}
Since $\bv_{q,b}$ is parallel to $\bn$, then $\bv_{q,b}\times\bn=0$
on $\pT$. In addition, the definition (\ref{2.5}) for the discrete
weak curl implies $\nabla_w \times \textbf{v}_q^{(2)}=0$ as
$$
(\nabla_w \times \textbf{v}_q^{(2)}, \varphi)_T = (0,
\curl\varphi)_T- h \langle \bv_{q,b}\times\bn,
\varphi\rangle_{\pT}=0 ,\quad \forall \ \varphi\in [P_{k-1}(T)]^3.
$$
Thus, it follows from (\ref{vq2-new}) and (\ref{November-30:500})
that
$$
\3bar \textbf{v}_q^{(2)} \3bar^2 \leqC h \left( \sum_{e\in \E_h^0}
\|\jump{q}\|_e^2 + \sum_{i=0}^m\|q-\bar q_i\|_{\Gamma_i}^2\right),
$$
which verifies the estimate (\ref{6.5}) for $\textbf{v}_q^{(2)}$.
This completes the proof of the lemma.
\end{proof}

\section{Error Analysis}
Based on the error equations shown as in Theorem
\ref{Thm:div-curl:theorem-error-eqns} and the {\em inf-sup}
condition in the previous section , we shall derive an estimate for
the error terms $\be_h$ and $\epsilon_h$ taken as the difference of
the WG finite element solution and the $L^2$ projection of the exact
solution.

\subsection{Some technical inequalities}

Assume that the finite element partition ${\cal T}_h$ of $\Omega$ is
shape regular as defined in \cite{wy1202}.  Let $T\in {\cal T}_h$ be
an element with $e$ as a face. The trace inequality holds true:
\begin{equation}\label{A4}
\|\psi\|_e^2\leqC\big(h_T^{-1}\|\psi\|_T^2+h_T\|\nabla
\psi\|_T^2\big),\qquad \forall \ \psi\in H^1(T).
\end{equation}
If $\phi$ is a polynomial, the inverse inequality holds true:
\begin{equation}\label{Div-Curl:inverse}
\|\nabla \phi\|_T \leqC h_T^{-1} \|\phi\|_T.
\end{equation}
From (\ref{A4}) and (\ref{Div-Curl:inverse}), we have
 \begin{equation}\label{Aa-trace}
\|\phi\|_e^2\leqC h_T^{-1}\|\phi\|_T^2.
\end{equation}

\begin{lemma}\label{lemmaA1}  \cite{wy1202} Let $k\ge 1$ be the order of the WG finite elements, and $1 \leq r \leq k$.
Let $\textbf{w}\in [H^{r+1}(\Omega)]^3$,  $\rho \in H^r (\Omega)$,
and $0 \leq m \leq 1$. There holds
\begin{align}\label{A1}
\sum_{T\in{\cal
T}_h}h_T^{2m}\|\textbf{w}-Q_0\textbf{w}\|^2_{T,m}&\leqC
h^{2(r+1)}\|\textbf{w}\|^2_{r+1},\\
\sum_{T\in{\cal T}_h}h_T^{2m}\|\nabla\times
\textbf{w}-\textbf{Q}_h(\nabla\times\textbf{w})\|^2_{T,m}&\leqC
h^{2r}\|\textbf{w}\|^2_{r+1},\label{A2}\\
\sum_{T\in{\cal T}_h}h_T^{2m}\|\rho-{\cal Q}_h\rho\|^2_{T,m}&\leqC
h^{2r}\|\rho\|^2_{r}.\label{A3}
\end{align}
\end{lemma}

In the WG finite element space $\bV_h$, we introduce a semi-norm as
follows
\begin{equation}\label{Norm-v1h}
|\bv|_{1,h} = \left(\sum_{T\in\T_h}
h_T^{-1}\|(\bv_0-\bv_b)\times\bn\|_{\pT}^2 +
h_T^{-1}\|(\mu\bv_0-\bv_b)\cdot\bn\|_{\pT}^2\right)^{\frac12}.
\end{equation}

\begin{lemma}\label{Lemma:Div-Curl:phi-estimate}
Assume that the finite element partition ${\cal T}_h$ of $\Omega$ is
shape regular as defined in \cite{wy1202} and $1\leq r \leq k$. Let
$\bw \in  [H^{r+1} (\Omega)]^3$ and $\rho \in H^r (\Omega)$. Then,
we have
\begin{align}\label{A7}
|s(Q_h\bw,\bv)| &\leqC h^r\|\bw\|_{r+1} \ |\bv|_{1,h},\\
|l_\bw(\bv)| & \leqC h^r\|\bw\|_{r+1}\ |\bv|_{1,h},\label{A8}\\
|\theta_\rho(\bv)| & \leqC h^r\|\rho\|_r\ |\bv|_{1,h},\label{A9}
\end{align}
for any $\bv \in \bV_h$. Here $l_\bw(\cdot)$ and
$\theta_\rho(\cdot)$ are defined in (\ref{l}) and (\ref{theta}).
\end{lemma}

\begin{proof}
Recall from (\ref{EQ:StabilityTerm}) that the stability term can be
decomposed into two parts:
$$
s(\bv,\bw) = s_1(\bv,\bw) + s_2(\bv, \bw),
$$
where
\begin{eqnarray}\label{EQ:StabilityTermOne}
s_1(\textbf{v},\textbf{w}) &=& \sum_{T\in {\cal T}_h}h_T^{-1}\langle
 (\textbf{v}_0-\textbf{v}_b)\times\bn,
(\textbf{w}_0-\textbf{w}_b)\times\bn\rangle_{\partial T},\\
s_2(\textbf{v},\textbf{w}) &=& \sum_{T\in {\cal T}_h}h_T^{-1}\langle
 (\mu\textbf{v}_0-\textbf{v}_b)\cdot\bn,
(\mu\textbf{w}_0-\textbf{w}_b)\cdot\bn\rangle_{\partial
T}.\label{EQ:StabilityTermTwo}
\end{eqnarray}

To prove (\ref{A7}), it suffices to derive the estimate (\ref{A7})
for $s_1(\cdot,\cdot)$ and $s_2(\cdot, \cdot)$ separately. To this
end, we use the Cauchy-Schwarz inequality, the trace
inequality~(\ref{A4}) and the estimate (\ref{A1}) to obtain
\begin{eqnarray*}
|s_1(Q_h\textbf{w},\textbf{v})| &=&\Big|\sum_{T\in {\cal
T}_h}h_T^{-1}\langle Q_0\textbf{w}\times
\bn-Q_b(\textbf{w}\times\bn), (\textbf{v}_0-
\textbf{v}_b)\times\bn\rangle_{\partial
T}\Big|\\
&=&\Big|\sum_{T\in {\cal T}_h}h_T^{-1}\langle
Q_0\textbf{w}\times\bn-\textbf{w}\times\bn,(\textbf{v}_0-
\textbf{v}_b)\times\bn\rangle_{\partial
T}\Big| \\
&\leq & \Big(\sum_{T\in {\cal T}_h}h_T^{-1}\|Q_0\textbf{w}-
\textbf{w}\|_{\partial T}^2\Big)^{\frac{1}{2}}\Big(\sum_{T\in {\cal
T}_h}h_T^{-1}\|(\textbf{v}_0- \textbf{v}_b)\times\bn\|_{\partial
T}^2\Big)^{\frac{1}{2}}\\
&\leqC &\Big(\sum_{T\in {\cal T}_h}h_T^{-2}\|Q_0\textbf{w}-
\textbf{w}\|_{T}^2+  |Q_0\textbf{w}- \textbf{w}
|_{1,T}^2\Big)^{\frac{1}{2}} \ |\textbf{v}|_{1,h}\\
&\leqC & h^r\|\textbf{w}\|_{r+1}\ |\textbf{v}|_{1,h}.
\end{eqnarray*}
To derive (\ref{A7}) for $s_2(\cdot, \cdot)$, we again use the
Cauchy-Schwarz inequality, the trace inequality~(\ref{A4}) and the
estimate (\ref{A1}) to obtain
\begin{eqnarray*}
|s_2(Q_h\textbf{w},\textbf{v})| &=&\Big|\sum_{T\in {\cal
T}_h}h_T^{-1}\langle \mu Q_0\textbf{w}\cdot
\bn-Q_b(\mu\textbf{w}\cdot\bn), (\mu\textbf{v}_0-
\textbf{v}_b)\cdot\bn\rangle_{\partial
T}\Big|\\
&=&\Big|\sum_{T\in {\cal T}_h}h_T^{-1}\langle
Q_0(\mu\textbf{w})\cdot\bn-\mu\textbf{w}\cdot\bn,(\mu\textbf{v}_0-
\textbf{v}_b)\cdot\bn\rangle_{\partial
T}\Big| \\
&\leq & \Big(\sum_{T\in {\cal T}_h}h_T^{-1}\|\mu(Q_0\textbf{w}-
\textbf{w})\|_{\partial T}^2\Big)^{\frac{1}{2}}\Big(\sum_{T\in {\cal
T}_h}h_T^{-1}\|(\mu\textbf{v}_0- \textbf{v}_b)\cdot\bn\|_{\partial
T}^2\Big)^{\frac{1}{2}}\\
&\leqC &\Big(\sum_{T\in {\cal T}_h}h_T^{-2}\|Q_0\textbf{w}-
\textbf{w}\|_{T}^2+  |Q_0\textbf{w}- \textbf{w}
|_{1,T}^2\Big)^{\frac{1}{2}} \ |\textbf{v}|_{1,h}\\
&\leqC & h^r\|\textbf{w}\|_{r+1}\ |\textbf{v}|_{1,h}.
\end{eqnarray*}

As to (\ref{A8}), we use the Cauchy-Schwarz inequality, the trace
inequality~(\ref{A4}) and the estimate (\ref{A2}) to obtain
\begin{equation*}
\begin{split}
|l_\textbf{w}(\textbf{v})|=&\left| \sum_{T\in{\cal T}_h}\langle
(\textbf{Q}_h-I)(\kappa\nabla \times\textbf{w}),
(\textbf{v}_0-\textbf{v}_b)\times \textbf{n})\rangle_{\partial
T}\right|\\
\leqC &\Big(\sum_{T\in{\cal T}_h}h_T \| (\textbf{Q}_h-I)(\nabla
\times\textbf{w}) \|^2_{\partial T}\Big)^{
\frac{1}{2}}\Big(\sum_{T\in{\cal T}_h}h_T^{-1}\|
 (\textbf{v}_0-\textbf{v}_b)\times\bn\|^2_{\partial T}\Big)^{ \frac{1}{2}}\\
\leqC & h^r\|\textbf{w}\|_{r+1}\ |\textbf{v}|_{1,h}.
\end{split}
\end{equation*}

Finally, we use the Cauchy-Schwarz inequality, the trace
inequality~(\ref{A4}) and the estimate (\ref{A3}) to obtain
\begin{equation*}
\begin{split}
|\theta_\rho(\textbf{v})|&=\left|\sum_{T\in{\cal T}_h}\langle \rho-
{\cal Q}_h \rho,
 (\mu\textbf{v}_0-\textbf{v}_b)\cdot\textbf{n}\rangle_{\partial T}\right|\\
&\leq   \Big(\sum_{T\in{\cal T}_h}h_T \| \rho- {\cal Q}_h
\rho\|^2_{\partial T}\Big)^{ \frac{1}{2}}\Big(\sum_{T\in{\cal
T}_h}h_T^{-1}\|
 (\mu\textbf{v}_0-\textbf{v}_b)\cdot\bn\|^2_{\partial T}\Big)^{ \frac{1}{2}}\\
 &\leqC h^r\| \rho\|_{r}\ |\textbf{v}|_{1,h}.
\end{split}
\end{equation*}
This completes the proof of the lemma.
\end{proof}

\subsection{Error estimates}

Recall that $\3bar\cdot\3bar_1$ defines a norm in the finite element
space $\bU_h^{0}$. This norm can be regarded as a discrete $H_0({\rm
curl})\cap H({\rm div}_\mu)$-norm under which the error function
$\be_h$ shall be measured.

\begin{theorem}\label{Thm:div-curl:theorem-error-estimate}
Assume that $k \geq 1$ be the order of the WG finite element scheme
(\ref{3.3-original})-(\ref{3.5-original}). Let $(\bu; p)\in
[H^{k+1}(\Omega)]^3\times H^k(\Omega)$ be the solution of the model
problem (\ref{EQ:model-problem-I}) and $(\bu_h; p_h)$ be the WG
finite element solution arising from
(\ref{3.3-original})-(\ref{3.5-original}). Then, we have
\begin{equation}\label{th1}
\3bar Q_h\textbf{u}-\textbf{u}_h\3bar_1+\|{\cal
Q}_hp-p_h\|_{W_h}\leqC h^k(\|\bu\|_{k+1}+\|p\|_k).
\end{equation}
\end{theorem}

\begin{proof} From Theorem
\ref{Thm:div-curl:theorem-error-eqns}, the error functions
$\be_h=Q_h\bu - \bu_h$ and $\epsilon_h={\cal Q}_h p - p_h$ satisfy
the equations
(\ref{EQ:div-curl:error-eq-01})-(\ref{EQ:div-curl:error-eq-02}). By
letting $\bv=\be_h$ in (\ref{EQ:div-curl:error-eq-01}) and then
using (\ref{EQ:div-curl:error-eq-02}), we obtain
\begin{equation}\label{November-29:001}
a(\be_h, \be_h) = \varphi_{\bu, p}(\be_h).
\end{equation}
The right-hand side can be estimated by using Lemma
\ref{Lemma:Div-Curl:phi-estimate} as follows
$$
|\varphi_{\bu,p}(\textbf{\be}_h)|\leqC
h^k(\|\bu\|_{k+1}+\|p\|_k)|\textbf{\be}_h|_{1,h}.
$$
Substituting the above into (\ref{November-29:001}) yields
\begin{equation*}
a(\be_h, \be_h) \leqC
h^k(\|\bu\|_{k+1}+\|p\|_k)|\textbf{\be}_h|_{1,h},
\end{equation*}
which, together with $|\textbf{\be}_h|_{1,h}^2 \leq a(\be_h,
\be_h)$, leads to
\begin{equation}\label{November-29:002}
a(\be_h, \be_h)^{1/2} \leqC h^k(\|\bu\|_{k+1}+\|p\|_k).
\end{equation}
Next, it follows from the equation (\ref{EQ:div-curl:error-eq-02})
that $\nabla_w\cdot(\mu\be_h)=0$. Thus,
$$
\3bar\be_h\3bar_1 = \3bar \be_h \3bar \leqC a(\be_h, \be_h)^{1/2}
$$
Combining the above inequality with (\ref{November-29:002}) gives
\begin{equation}\label{November-29:003}
\3bar\be_h\3bar_1 \leqC h^k(\|\bu\|_{k+1}+\|p\|_k).
\end{equation}

The error function $\epsilon_h$ can be estimated by using the {\em
inf-sup} condition derived in Lemma \ref{lvq2}. To this end, from
the equation (\ref{EQ:div-curl:error-eq-01}), we have
\begin{equation}\label{November:30:finenow}
b(\bv,\epsilon_h) = \varphi_{\bu,p}(\bv) - a(\be_h,\bv).
\end{equation}
By letting $\bv = \bv_{\epsilon_h}$ in (\ref{November:30:finenow}),
we arrive at
$$
\|\epsilon_h\|_{W_h}^2 \leqC |\varphi_{\bu,p}(\bv_{\epsilon_h})| +
|a(\be_h,\bv_{\epsilon_h})|.
$$
Using Lemma \ref{Lemma:Div-Curl:phi-estimate} and the error estimate
(\ref{November-29:002}) we obtain
$$
\|\epsilon_h\|_{W_h}^2 \leqC h^k (\|\bu\|_{k+1}+\|p\|_k) \3bar
\bv_{\epsilon_h}\3bar,
$$
which, together with (\ref{6.5}), leads to
$$
\|\epsilon_h\|_{W_h} \leqC h^k (\|\bu\|_{k+1}+\|p\|_k).
$$
This completes the proof of the theorem.
\end{proof}

The mesh-dependent norm $\|q\|_{W_h}$ in the finite element space
$W_h$ is a scaled discrete $H^1$ norm for piecewise smooth
functions. For the lowest order WG element (i.e., piecewise linear
for $\bu$ and piecewise constant for $p$), the error estimate in
Theorem \ref{Thm:div-curl:theorem-error-estimate} does not give any
convergence for the approximation of $p$. However, it is possible to
replace the $\|\cdot\|_{W_h}$ norm by the standard $L^2$ norm in the
error estimate (\ref{th1}) if the solution of the following second
order elliptic problem is $H^2$-regular:
\begin{equation}\label{EQ:2015-DualOne}
\begin{split}
-\nabla\cdot(\mu \nabla \Phi) & = \epsilon_h,\qquad \mbox{in} \ \Omega,\\
\Phi|_{\Gamma_i} & = \alpha_i,\qquad i=0,1,\ldots,m,
\end{split}
\end{equation}
where $\alpha_0=0$ and $\alpha_i$ is a set of constants such that
$\langle \nabla\Phi\cdot\bn_i, 1\rangle_{\Gamma_i}=0$ for
$i=1,\ldots, m$. In fact, for the solution of
(\ref{EQ:2015-DualOne}), it can be seen that $\nabla\Phi\in
\mathds{Y}_\mu(\Omega)\cap H({\rm div}_\mu;\Omega)$. Moreover, the
projection $\bz=Q_h(\nabla\Phi)$ (see (\ref{EQ:12-31:001}) for its
definition) is a finite element function in $\bU_h^0$. The desired
error estimate for $\epsilon_h$ in $L^2(\Omega)$ can then be
obtained by taking $\bv=\bz$ in the error equation
(\ref{EQ:div-curl:error-eq-01}). Details are left to interested
readers as an exercise.

\subsection{$L^2$-error estimates}
To derive an $L^2$ error estimate for $\be_h$, we consider the dual
problem of seeking $\bpsi\in H_0({\rm curl};\Omega) \cap H({\rm
div}_\mu;\Omega)$ and $\tau\in L^2(\Omega)$ such that
\begin{equation}\label{EQ:model-problem-I-Dual}
\begin{array}{rll}
(\kappa \curl\bpsi,\curl\bv) + (\nabla\cdot(\mu\bv), \tau)=&
(\be_0,\bv), &\ \forall\ \bv\in
\mathds{Y}_\mu(\Omega)\cap H({\rm div}_\mu;\Omega),\\
(\nabla\cdot(\mu\bpsi), w) =& 0, &\ \forall\ w\in L^2(\Omega),
\\
\langle \mu\bpsi\cdot\bn_i, 1\rangle_{\Gamma_i} =&0, &\ i=1,\ldots,
m.
\end{array}
\end{equation}
Assume that the dual problem (\ref{EQ:model-problem-I-Dual}) has the
$[H^{2} (\Omega)]^3 \times  H^1(\Omega)$-regularity property in the
sense that the solution $(\bpsi;\tau)\in [H^{2}(\Omega)]^3 \times
H^1(\Omega)$ and satisfies the following a priori estimate:
\begin{equation}\label{7.2}
\|\bpsi\|_{2}+\|\tau\|_1\leqC \|\textbf{e}_0\|_0.
\end{equation}

By using a Lagrange multiplier $\gamma=(\gamma_1,\ldots,
\gamma_m)\in \mathbb{R}^m$, the dual problem
(\ref{EQ:model-problem-I-Dual}) can be rewritten in an equivalent
form as follows. Find $\bpsi\in H_0({\rm curl};\Omega) \cap H({\rm
div}_\mu;\Omega)$, $\tau\in L^2(\Omega)$, and $\gamma\in \bbR^m$
such that
\begin{equation}\label{EQ:model-problem-I-Dual-new}
\left\{
\begin{array}{c}\displaystyle
(\kappa \curl\bpsi,\curl\bv) + (\nabla\cdot(\mu\bv), \tau)
+\sum_{i=1}^m
\langle \mu\bv\cdot\bn_i, \gamma_i\rangle_{\Gamma_i}= (\be_0,\bv),\\
\displaystyle(\nabla\cdot(\mu\bpsi), w) + \sum_{i=1}^m \langle
\mu\bpsi\cdot\bn_i, s_i\rangle_{\Gamma_i} = 0,
\end{array}
\right.
\end{equation}
for all $\bv\in H_0({\rm curl};\Omega)\cap H({\rm div}_\mu;\Omega)
,\ w\in L^2(\Omega),$ and $s\in \bbR^m$.

\begin{theorem}\label{Thm:Div-Curl:L2-error-estimate}
Let $k\geq 1$ be the order of the WG scheme
(\ref{3.3-original})-(\ref{3.5-original}). Let $(\textbf{u}; p) \in
[H^{k+1}(\Omega)]^3 \times H^{k}(\Omega)$ and $(\textbf{u}_h; p_h)
\in \bV_h \times W_h$ be the solutions of the problem
(\ref{EQ:model-problem-I}) and
(\ref{3.3-original})-(\ref{3.5-original}), respectively. Then, the
following estimate holds true
\begin{equation}\label{7.3}
\|Q_0\textbf{u}-\textbf{u}_0\|\leqC
h^{k+1}\big(\|\textbf{u}\|_{k+1}+\|p\|_k\big).
\end{equation}
\end{theorem}

\begin{proof}
From the first equation of (\ref{EQ:model-problem-I-Dual-new}), we
see that the the equation (\ref{5.2}) is satisfied by
$(\bw;\rho)=(\bpsi;\tau)$ with $\eta=\be_0$. In addition, the
boundary condition (\ref{November.28.100}) is verified by
$$
\tau|_{\Gamma_0}=0,\ \tau|_{\Gamma_i}=-\gamma_i,\ i=1,\ldots, m.
$$
Thus, by using (\ref{Div-Curl:Feb9:300-new}) in Lemma \ref{lemma5.1}
with $\bv=\be_h\in \bU_h^{0}$, we obtain
\begin{equation}\label{7.4}
\begin{split}
\|Q_0\textbf{u}-\textbf{u}_0\|^2 = & (\kappa\nabla_w \times
(Q_h\bpsi),\nabla_w \times\be_h)_h \\
& +(\nabla_w\cdot(\mu\be_h),{\cal
Q}_h\tau)_h-\theta_\tau(\be_h)-l_\bpsi(\be_h).
\end{split}
\end{equation}
Note that the error equation (\ref{EQ:div-curl:error-eq-02}) implies
$\nabla_w\cdot(\mu\textbf{e}_h)=0$ so that (\ref{7.4}) can be
simplified as
\begin{equation}\label{7.7old}
\|Q_0\textbf{u}-\textbf{u}_0\|^2=a( Q_h\bpsi, \textbf{e}_h)
-\varphi_{\bpsi,\tau}(\textbf{e}_h),
\end{equation}
where
$$
\varphi_{\bpsi,\tau}(\textbf{e}_h)=\theta_\tau(\textbf{e}_h)+
l_\bpsi(\textbf{e}_h)+s( Q_h\bpsi, \textbf{e}_h ).
$$
From (\ref{4.4}) and the second equation of
(\ref{EQ:model-problem-I-Dual}), we have
\begin{equation*}
\begin{split}
b(Q_h\bpsi,\epsilon_h)&=(\nabla_w \cdot (\mu Q_h\bpsi),\epsilon_h)_h\\
&= ({\cal Q}_h\nabla\cdot(\mu\bpsi), \epsilon_h)_h =0,
\end{split}
\end{equation*}
which, together with (\ref{EQ:div-curl:error-eq-01}) and
(\ref{7.7old}), leads to
\begin{equation}\label{7.7}
\begin{split}
\|Q_0\textbf{u}-\textbf{u}_0\|^2 &=a(\be_h,
Q_h\bpsi)+b(Q_h\bpsi,\epsilon_h)
-\varphi_{\bpsi,\tau}(\textbf{e}_h)\\
&=\varphi_{\textbf{u},p}(Q_h\bpsi)-\varphi_{\bpsi,\tau}(\textbf{e}_h).
\end{split}
\end{equation}

The two terms on the righ-hand side of (\ref{7.7}) can be estimated
as follows. First, by letting $r=1$, $\bv=\be_h$, and
$(\bw;\rho)=(\bpsi;\tau)$ in Lemma
\ref{Lemma:Div-Curl:phi-estimate}, we obtain
\begin{equation}\label{7.8}
\begin{split}
|\varphi_{\bpsi,\tau}(\textbf{e}_h)| \leqC
h(\|\bpsi\|_{2}+\|\tau\|_1)\ |\textbf{e}_h|_{1,h} \leqC h
\|\textbf{e}_0\|\ \3bar\be_h\3bar,
\end{split}
\end{equation}
where we have used (\ref{7.2}) in the second inequality. Next, by
letting $r=k$, $\bv=Q_h\bpsi$, and $(\bw;\rho)=(\bu;p)$ in Lemma
\ref{Lemma:Div-Curl:phi-estimate}, we arrive at
\begin{equation}\label{7.8:100}
|\varphi_{\bu,p}(Q_h\bpsi)|\leqC h^k(\|\bu\|_{k+1}+\|p\|_k)\
|Q_h\bpsi|_{1,h}.
\end{equation}

To estimate $|Q_h\bpsi|_{1,h}$, we recall from (\ref{Norm-v1h}) and
the definition (\ref{EQ:12-31:001}) of $Q_h$ that
\begin{equation}\label{EQ:2015:Jan-2:001}
|Q_h\bpsi|_{1,h}^2 = \sum_{T\in\T_h} h_T^{-1}\|(Q_0\bpsi -
\mathds{Q}_b\bpsi)\times\bn\|_{\partial T}^2 + h_T^{-1}\|(\mu
Q_0\bpsi - \mathds{Q}_b\bpsi)\cdot\bn\|_{\partial T}^2.
\end{equation}
By using (\ref{EQ:12-31:002}), we have
\begin{eqnarray*}
(\mathds{Q}_b\bpsi)\times\bn &=&
Q_b(\bn\times(\bpsi\times\bn))\times\bn\\
&=&Q_b\bpsi\times\bn,\\
(\mathds{Q}_b\bpsi)\cdot\bn &=& Q_b(\mu\bpsi\cdot\bn).
\end{eqnarray*}
Thus,
\begin{eqnarray*}
(Q_0\bpsi - \mathds{Q}_b\bpsi)\times\bn &=& (Q_0\bpsi -
{Q}_b\bpsi)\times\bn, \\
(\mu Q_0\bpsi - \mathds{Q}_b\bpsi)\cdot\bn &=& \mu(Q_0
\bpsi-Q_b\bpsi)\cdot\bn.
\end{eqnarray*}
Substituting the above identities into (\ref{EQ:2015:Jan-2:001})
yields
\begin{equation*}\label{EQ:2015:Jan-2:002}
\begin{split}
|Q_h\bpsi|_{1,h}^2 &= \sum_{T\in\T_h} h_T^{-1}\|(Q_0\bpsi -
{Q}_b\bpsi)\times\bn\|_{\partial T}^2 + h_T^{-1}\|\mu(Q_0\bpsi -
{Q}_b\bpsi)\cdot\bn\|_{\partial T}^2\\
&\leqC  \sum_{T\in\T_h} h_T^{-1}\|Q_0\bpsi - {Q}_b\bpsi\|_{\partial
T}^2\\
&\leq \sum_{T\in\T_h} h_T^{-1}\|Q_0\bpsi -
\bpsi\|_{\partial T}^2\\
&\leqC\sum_{T\in\T_h}(h_T^{-2}\|\bpsi-Q_0\bpsi\|_T^2 +
 \|\nabla(\bpsi-Q_0\bpsi)\|_T^2) \\
&\leqC\sum_{T\in\T_h}h_T^2\|\nabla^2\bpsi\|_T^2 \\
&\leqC h^2 \|\bpsi\|_2^2,
\end{split}
\end{equation*}
where we have used the $L^2$ property of $Q_b$ in line 3, the trace
inequality (\ref{A4}) in line 4, and the usual approximation
property for the $L^2$ projection operator $Q_0$ in line 5.
Inserting the above estimate into (\ref{7.8:100}) and then using
(\ref{7.2}), we obtain
\begin{equation}\label{7.8:200}
|\varphi_{\bu,p}(Q_h\bpsi)|\leqC h^{k+1}(\|\bu\|_{k+1}+\|p\|_k)\
\|\be_0\|.
\end{equation}
Now, combining (\ref{7.7}) with (\ref{7.8}) and (\ref{7.8:200}), we
arrive at
\begin{equation*}
\begin{split}
\|Q_0\textbf{u}-\textbf{u}_0\|^2  \leqC \Big(h \3bar\be_h\3bar +
h^{k+1}(\|\bu\|_{k+1}+\|p\|_k)\Big)\ \| \textbf{e}_0\|.
\end{split}
\end{equation*}
It follows that
\begin{equation*}
\|Q_0\textbf{u}-\textbf{u}_0\| \leqC h \3bar\be_h\3bar +
h^{k+1}(\|\bu\|_{k+1}+\|p\|_k),
\end{equation*}
which, together with Theorem
\ref{Thm:div-curl:theorem-error-estimate} and the fact that
$\3bar\be_h\3bar\leq \3bar\be_h\3bar_1$, completes the proof of
the theorem.
\end{proof}

The WG finite element solution $\bu_h=\{\bu_0; \bu_b\}$ consists of
two components on each element $T\in\T_h$: (1) the element interior
value $\bu_0$, and (2) the element boundary value $\bu_b$. The rest
of this section is devoted to an error analysis for $\bu_b$. To this
end, we introduce the following topology in the finite element space
$\bV_h$
$$
\|\bv_b\|_{\E_h}=\Big(\sum_{T\in {T}_h} h_T \int_{\partial T}
|\bv_b|^2 ds\Big)^{\frac{1}{2}},\qquad \bv=\{\bv_0; \bv_b\}\in
\bV_h.
$$
It is clear that the above defines an $L^2$-like norm for the face
variable $\bv_b$.

\begin{theorem}\label{Thm:Div-Curl:L2-error-estimate-ub}
Let $k\geq 1$ be the order of the WG scheme
(\ref{3.3-original})-(\ref{3.5-original}). Let $(\textbf{u}; p) \in
[H^{k+1} (\Omega)]^3 \times H^{k}(\Omega)$ and $(\textbf{u}_h; p_h)
\in \bV_h \times W_h$ be the solution of the problem
(\ref{EQ:model-problem-I}) and
(\ref{3.3-original})-(\ref{3.5-original}), respectively. Then, we
have
\begin{equation}\label{EQ:div-curl:feb11:800}
\|\mathds{Q}_b\textbf{u}-\textbf{u}_b\|_{\E_h}\leqC
h^{k+1}\big(\|\textbf{u}\|_{k+1}+\|p\|_k\big).
\end{equation}
Here $\mathds{Q}_b\textbf{u}$ is the projection of $\bu$ on each
face $e\in\E_h$ given by (\ref{EQ:12-31:002}).
\end{theorem}

\begin{proof} Note that $\mathds{Q}_b\bu-\bu_b=\be_b$ is the error
of the WG finite element solution on $\E_h$ -- the set of all
element faces in $\T_h$. It follows from the triangle inequality
that on each $\pT$
\begin{equation*}
\begin{split}
|\be_b|^2 =&\ |\be_b\cdot\bn|^2 + |\be_b\times\bn|^2\\
\leq & \ 2 |(\mu\be_0-\be_b)\cdot\bn|^2 + 2 |(\be_0-\be_b)\times\bn|^2\\
 & \ + 2 |\mu\be_0\cdot\bn|^2 + 2|\be_0\times\bn|^2.
\end{split}
\end{equation*}
Summing over all the elements and then using (\ref{Eq:bar1-norm})
yields
\begin{eqnarray*}
\sum_{T\in\T_h} h_T^{-1} \int_\pT |\be_b|^2ds &&\leqC \3bar \be_h
\3bar_1^2 + \sum_{T\in\T_h} h_T^{-1} \int_\pT (|\mu\be_0\cdot\bn|^2
+|\be_0\times\bn|^2)ds\\
&& \leqC \3bar \be_h \3bar_1^2 + h^{-2}\|\be_0\|^2,
\end{eqnarray*}
where we have used the trace inequality (\ref{Aa-trace}) in the
second line. The last inequality further leads to the following
estimate:
\begin{eqnarray*}
\sum_{T\in\T_h} h_T \int_\pT |\be_b|^2ds \leqC h^2 \3bar \be_h
\3bar_1^2 + \|\be_0\|^2,
\end{eqnarray*}
which, together with the error estimates (\ref{th1}) and
(\ref{7.3}), implies the desired estimate
(\ref{EQ:div-curl:feb11:800}).
\end{proof}

\end{document}